\documentclass[12pt, a4paper]{amsart}

\usepackage{amsfonts, amsmath, amssymb,latexsym}
\usepackage[english]{babel}
\usepackage[utf8]{inputenc}
\usepackage{url}
\usepackage[top=0.5in, left=1in, bottom=1in]{geometry}

\usepackage{mathptmx}
\usepackage{times}

\usepackage{graphicx}        

\usepackage{amsmath,amscd, amsfonts}

\newtheorem{thm}{Theorem}[section]

\newtheorem{prop}[thm]{Proposition}
\newtheorem{defn}[thm]{Definition}

\newcommand{\Cl}{\mathop{\mathrm{Cl}}}

\usepackage{lineno}
\usepackage{tikz}
\usepackage{pgfplots}
\usepgflibrary{shapes.geometric} 
\usepgflibrary[shapes.geometric] 
\usetikzlibrary{shapes.geometric} 
\usetikzlibrary[shapes.geometric] 

{\begin{list}{}{\setlength{\leftmargin}{18pt}\setlength{\rightmargin}{12pt}}
 \item[]\ignorespaces\small}
 {\unskip\end{list}}

\newcommand{\exclude}[1]{} %

\begin{document}

\title[The  $E_8$ geometry from a Clifford perspective]
{The  $E_8$ geometry from a Clifford perspective}


\maketitle

\vspace{-24pt}

\begin{center}
\author{{\bf Pierre-Philippe Dechant}}

\small

Departments of Mathematics and Biology\\
York Centre for Complex Systems Analysis\\
University of York, Heslington YO10 5GE, United Kingdom\\
ppd22@cantab.net [presenter, corresponding]
\end{center}

\thispagestyle{empty}



\begin{abstract}
This paper considers the geometry of $E_8$ from a Clifford point of view in three complementary ways.
Firstly, in earlier work, I had shown how to construct the four-dimensional exceptional root systems from the 3D root systems using Clifford techniques, by constructing them in the 4D even subalgebra of the 3D Clifford algebra; for instance the icosahedral root system $H_3$ gives rise to the largest (and therefore exceptional) non-crystallographic root system $H_4$. 
Arnold's trinities and the McKay correspondence then hint that there might be an indirect connection between the icosahedron and $E_8$. 
Secondly, in a related construction, I have now made this connection explicit for the first time:
in the 8D Clifford algebra of 3D space the $120$ elements of the icosahedral group $H_3$ are doubly covered by $240$ 8-component objects, which endowed with a `reduced inner product' are exactly the $E_8$ root system.
It was previously known that $E_8$ splits into $H_4$-invariant subspaces, and we discuss the folding construction relating the two pictures.
This folding is a partial version of the one used for the construction of the Coxeter plane, so thirdly we discuss the geometry of the Coxeter plane in a Clifford algebra framework.
We advocate the complete factorisation of the Coxeter versor in the Clifford algebra into exponentials of bivectors describing rotations in orthogonal planes with the rotation angle giving the correct exponents, which gives much more geometric insight than the usual approach of complexification and search for complex eigenvalues.
In particular, we explicitly find these factorisations for the 2D, 3D and 4D root systems, $D_6$ as well as $E_8$, whose Coxeter versor factorises as $W=\exp(\frac{\pi}{30}B_C)\exp(\frac{11\pi}{30}B_2)\exp(\frac{7\pi}{30}B_3)\exp(\frac{13\pi}{30}B_4)$.
This explicitly describes 30-fold rotations in 4 orthogonal planes with the correct exponents $\{1, 7, 11, 13, 17, 19, 23, 29\}$ arising completely algebraically from the factorisation.

\end{abstract}

\begin{figure}
	\begin{center}

\includegraphics[width=12cm]{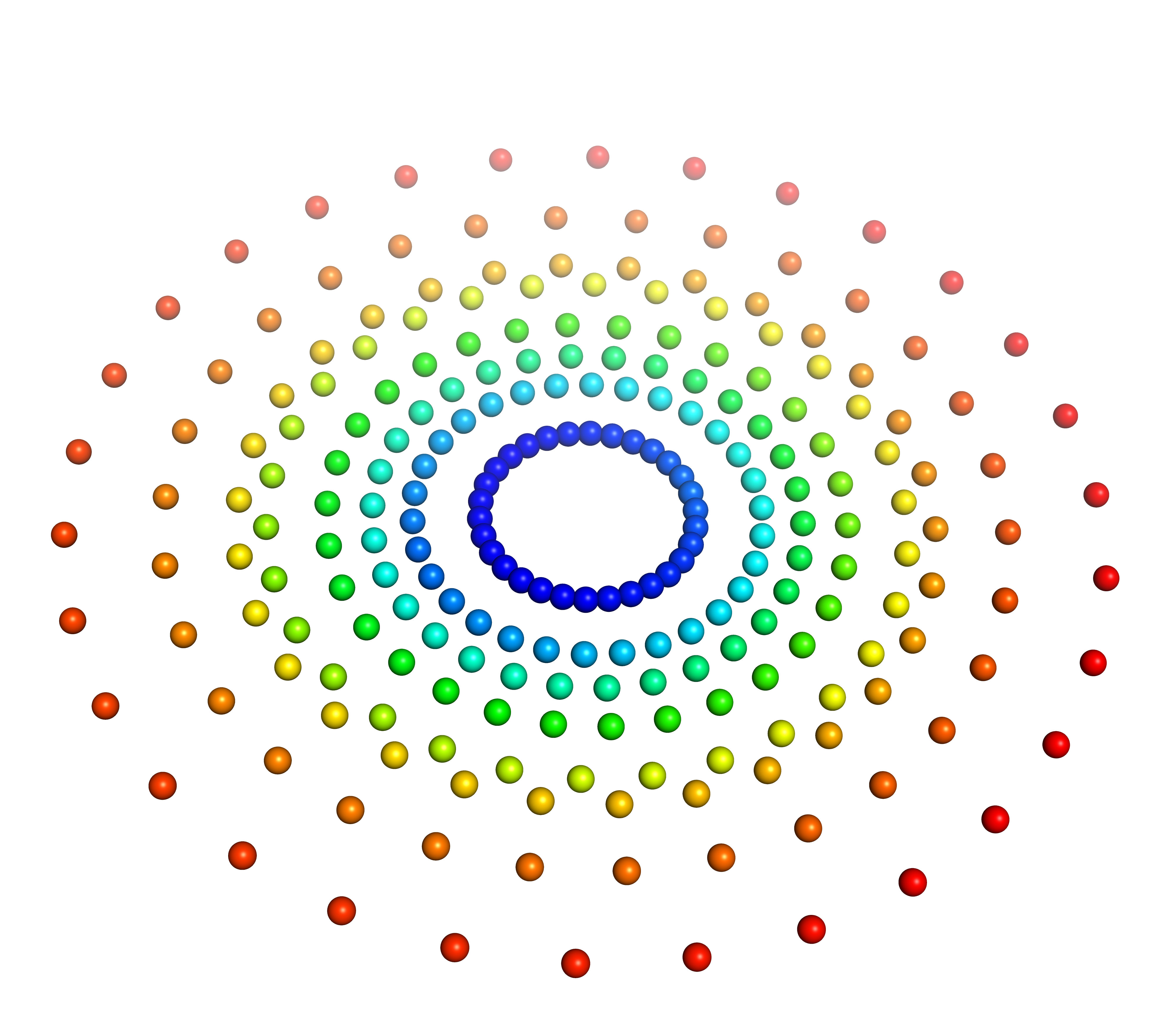}

\end{center}
\caption[$E_6^+$]{The projection of the $240$ 8D root vectors in the $E_8$ root system into the Coxeter plane is a popular visualisation of $E_8$ as 8 concentric circles of $30$. }
\label{figE8CoxPl}
\end{figure}

\section{Introduction}

At the last AGACSE in 2012, I presented a new link between the geometries of three and four dimensions; 
in particular, I have shown that any reflection group in three dimensions induces a corresponding symmetry group in four dimensions, via a new Clifford spinor construction \cite{Dechant2012AGACSE,Dechant2012Induction}.
This connection had been overlooked for centuries (usually one assumes the larger groups are more fundamental) but the new construction derives all the exceptional phenomena in 4D -- $D_4$, $F_4$ and $H_4$ -- via induction from the 3D symmetry groups of the Platonic solids $A_3$, $B_3$ and $H_3$. This spinor construction also explains the unusual 4D automorphism groups. The 4D groups in fact do not contain anything that is not already present in the 3D groups they are induced from \cite{Dechant2012CoxGA}.

This begs the obvious question of whether one can derive the largest exceptional geometry $E_8$, the holy grail of mathematics and physics, in a Clifford algebra approach as well.  I have now found a construction that constructs the $E_8$ root system (see Fig. \ref{figE8CoxPl}) in analogy to the above construction going 
from 3D to 4D \cite{Dechant2016Birth}. This previous construction worked along the following lines: each 3D root system -- which generates the corresponding reflection symmetry group via the reflections in the hyperplanes orthogonal to the root vectors -- allows one to form a group of spinors by multiplying together even numbers of the reflection generating root vectors in the Clifford algebra; these spinors have four components (the usual $1$ scalar and $3$ bivector components) and one can endow these with a 4D Euclidean metric. One can then show that the resulting set of spinors reinterpreted as 4D vectors satisfies the axioms for a 4D root system, thereby generating a symmetry group in 4D. 

For instance, starting with the root system $H_3$ which generates the symmetries of the icosahedron, one generates a group $2I$ of $120$ spinors (the binary icosahedral group) via multiplication in the Clifford algebra. These are precisely the $120$ root vectors of the 4D root system $H_4$ once reinterpreted using the 4D Euclidean metric. $H_4$ is exceptional and the largest non-crystallographic Coxeter group; it also has the exceptional automorphism group $2I\times 2I$ \cite{Koca2006F4,Koca2003A4B4F4}. However, this is trivial to see in terms of the spinor group $2I$, as it must be trivially closed under left and right multiplication via group closure.

Taking the full set of pinors (i.e. not restricting to even products of root vectors) of the icosahedron, one generates a group (doubly covering the $120$ elements of $H_3$) of $240$ pinors with 8 components, as befits a group of multivectors in 3D. I have now been able to show that these $240$ pinors are precisely the $240$ roots of $E_8$ by using a reduced inner product, thereby generating the exceptional group $E_8$ \cite{Dechant2016Birth}. 

It is extraordinary that the exceptional geometry $E_8$ has been hiding in the shadows of icosahedral geometry for millennia, without anyone noticing. As with the 4D induction construction, this discovery seems only possible in Clifford algebra. There is much prejudice against the usefulness of Clifford algebras (since they have matrix representations) and usually matrix methods are equivalent if less insightful -- but the 4D and 8D induction constructions are to my knowledge the only results that {require} Clifford algebra and were completely invisible to standard matrix methods.

This paper is structured in the following way. After some background and basic definitions  in Section \ref{background}, we summarise the general result that   any 3D reflection group induces a 4D symmetry group, via their root systems (Section \ref{sec_pin}).  In terms of their 3D spinors,  the Platonic root systems $(A_3, B_3, H_3)$  thus induce all the exceptional 4D root systems $(D_4, F_4, H_4)$ with this spinorial approach also explaining their unusual symmetry groups.  We treat as an illustrative example the case of the icosahedral group $H_3$ whose spinor group is the binary icosahedral group $2I$.  It doubly covers the $60$ icosahedral rotations in terms of $120$ spinors, but when thought of as a collection of 4-component objects is precisely $H_4$, the exceptional largest non-crystallographic  root system. Arnold's trinities contain the above Platonic root systems $(A_3, B_3, H_3)$ as well as the exceptional Lie groups $(E_6, E_7, E_8)$. They are connected indirectly via various intermediate trinities as well as more explicitly via the new spinor construction in combination with the McKay correspondence. This therefore hints that the icosahedron may be indirectly connected with $E_8$. Section \ref{birth} makes this explicit for the first time via  a  new, direct construction  \cite{Dechant2016Birth}, which explicitly constructs the $240$ roots  $E_8$  from the $240$ pinors that doubly cover the $120$ elements of the icosahedral group in the Clifford algebra of 3D. Thus \textbf{all} the exceptional root systems can  be seen to be contained in the geometry of 3D via the Platonic symmetries and the Clifford algebra of 3D. This offers a completely new way of viewing these in terms of spinorial geometry. Section \ref{sec_spin} discusses $H_4$ as a rotational subgroup of $E_8$, and thus affords a second way of viewing $H_4$ more naturally as a group of rotations rather than reflections (as is the standard view in the Coxeter picture). This construction of $H_4$ from $E_8$ is in fact a partial folding or 4-fold colouring of the $E_8$ diagram, whilst a complete folding leads to the Coxeter plane. We therefore finish by discussing the geometry of the Coxeter plane of various root systems, notably $E_8$, and point out which advantages a Clifford algebraic view has to offer over the naive standard view which involves complexifying the real geometry: in Clifford algebra the complex eigenvalues in fact arise naturally as rotations in mutually orthogonal eigenplanes of the Coxeter element (Section \ref{Coxeterplane}) with the bivectors of the planes providing the relevant complex structures; we conclude in Section \ref{concl} and again stress that for root systems and reflection groups Clifford algebra is the most natural framework.

\section{Background}\label{background}
Lie groups are a ubiquitous in mathematics and physics. In particular, the largest exceptional Lie group $E_8$ is of fundamental importance in String Theory and Grand Unified Theories and is thus arguably the single most important symmetry group in modern theoretical physics. Lie groups are continuous (group manifolds) but they are closely related to their corresponding Lie algebras whose  non-trivial part in turn is described by a root system: 

\begin{defn}[Root system] \label{DefRootSys}
A \emph{root system} is a collection $\Phi$ of non-zero (root)  vectors $\alpha$ spanning an $n$-dimensional Euclidean vector space $V$ endowed with a positive definite bilinear form, which satisfies the  two axioms:
\begin{enumerate}
\item $\Phi$ only contains a root $\alpha$ and its negative, but no other scalar multiples: $\Phi \cap \mathbb{R}\alpha=\{-\alpha, \alpha\}\,\,\,\,\forall\,\, \alpha \in \Phi$. 
\item $\Phi$ is invariant under all reflections corresponding to root vectors in $\Phi$: $s_\alpha\Phi=\Phi \,\,\,\forall\,\, \alpha\in\Phi$. 
The reflection $s_\alpha$ in the hyperplane with normal $\alpha$ is given by $$s_\alpha: \lambda\rightarrow s_\alpha(\lambda)=\lambda - 2\frac{(\lambda|\alpha)}{(\alpha|\alpha)}\alpha,$$\label{refl} where $(\cdot \vert \cdot)$ denotes the inner product on $V$.
\end{enumerate}
\end{defn}

For a crystallographic root system, a subset $\Delta$ of $\Phi$, called \emph{simple roots} $\alpha_1, \dots, \alpha_n$, is sufficient to express every element of $\Phi$ via $\mathbb{Z}$-linear combinations with coefficients of the same sign. 
$\Phi$ is therefore  completely characterised by this basis of simple roots. In the case of the non-crystallographic root systems $H_2$, $H_3$ and $H_4$, the same holds for the extended integer ring $\mathbb{Z}[\tau]=\lbrace a+\tau b| a,b \in \mathbb{Z}\rbrace$, where $\tau$ is   the golden ratio $\tau=\frac{1}{2}(1+\sqrt{5})=2\cos{\frac{\pi}{5}}$, and $\sigma$ is its Galois conjugate $\sigma=\frac{1}{2}(1-\sqrt{5})$ (the two solutions to the quadratic equation $x^2=x+1$).  
 For the crystallographic root systems, the classification in terms of Dynkin diagrams essentially follows the one familiar from Lie groups and Lie algebras, as their Weyl groups are  the crystallographic Coxeter groups. A mild generalisation to so-called Coxeter-Dynkin diagrams is necessary for the non-crystallographic root systems:
\begin{defn}[Coxeter-Dynkin diagram and Cartan matrix] 
	A graphical representation of the geometric content of a root system is given by \emph{Coxeter-Dynkin diagrams}, where nodes correspond to simple roots, orthogonal roots are not connected, roots at $\frac{\pi}{3}$ have a simple link, and other angles $\frac{\pi}{m}$ have a link with a label $m$. 
The  \emph{Cartan matrix} of a set of simple roots $\alpha_i\in\Delta$ is defined as the matrix
	$A_{ij}=2{(\alpha_i\vert \alpha_j)}/{(\alpha_j\vert \alpha_j)}$.
\end{defn}
For instance, the root system of the icosahedral group $H_3$ has one link labelled by $5$ (via the above relation $\tau=2\cos{\frac{\pi}{5}}$), as does its four-dimensional analogue $H_4$.

The reflections in the second axiom of the root system generate a reflection group. A Coxeter group is a mathematical abstraction of the concept of a reflection group via  involutive  generators (i.e. they square to the identity, which captures the idea of a reflection), subject to mixed relations that represent $m$-fold rotations (since two successive reflections generate a rotation in the plane defined by the two roots).
\begin{defn}[Coxeter group] A \emph{Coxeter group} is a group generated by a set of involutive generators $s_i, s_j \in S$ subject to relations of the form $(s_is_j)^{m_{ij}}=1$ with $m_{ij}=m_{ji}\ge 2$ for $i\ne j$. 
\end{defn}
The  finite Coxeter groups have a geometric representation where the involutions are realised as reflections at hyperplanes through the origin in a Euclidean vector space $V$, i.e. they are essentially  just the classical reflection groups. In particular, then the abstract generator $s_i$ corresponds to the simple {reflection}
$s_i: \lambda\rightarrow s_i(\lambda)=\lambda - 2\frac{(\lambda|\alpha_i)}{(\alpha_i|\alpha_i)}\alpha_i$
 in the hyperplane perpendicular to the  simple {root } $\alpha_i$.
The action of the Coxeter group is  to permute these root vectors, and its  structure is thus encoded in the collection  $\Phi\in V$ of all such roots, which in turn form a root system.

It is thus straightforward to move between those four related concepts of Lie groups/algebras, root systems and Coxeter groups and we will usually not make a distinction -- with the exception of non-crystallographic root systems such as  $H_3$ (which generates icosahedral symmetry) and its 4D analogue $H_4$: their non-crystallographic nature means that there is no associated Lie algebra, since going to the Lie algebra level in the Kac-Moody approach the Cartan matrix entries appear in the Chevalley-Serre relations as powers of the generators, which have to be integers rather than $\mathbb{Z}[\tau]$ integers.  

The $E_8$ root system is thus usually thought of as an exceptional (i.e. there are no corresponding symmetry groups in arbitrary dimensions) phenomenon of  eight-dimensional geometry, with no connection to the 3D geometry we inhabit. Surprisingly, the eight dimensions of 3D Clifford algebra actually allow $E_8$ to fit into our 3D geometry; likewise all 4D exceptional root systems arise within 3D geometry from the Platonic symmetries, unveiling them all as intrinsically 3D phenomena. This opens up a revolutionary way of viewing exceptional higher-dimensional geometries in terms of 3D spinorial geometry.

We employ a Clifford algebra framework, which via the geometric product by $xy=x\cdot y+x \wedge y$ affords a uniquely simple prescription for performing reflections $\lambda\rightarrow s_i(\lambda)=\lambda - 2(\alpha \cdot\lambda) \alpha =-\alpha \lambda \alpha$ (assuming unit normalisation). This yields a double cover of the reflections (since $\alpha$ and $-\alpha$ give the same reflection) and thus via the Cartan-Diedonn\'e theorem (which allows one to express any orthogonal transformation as the product of reflections) one gets a double cover  of all orthogonal transformations $\underbar{A}: v\rightarrow  v'=\underbar{A}(v)=\pm{\tilde{A}vA}$, in spaces of any dimension and signature, by products of unit vectors $A=\alpha_1\alpha_2\dots \alpha_k$ called versors. We also call even versors spinors, as they doubly cover the special orthogonal transformations, and general versors pinors, as they doubly cover the  orthogonal transformations.   

The non-exceptional Lie groups have been realised as spin groups  in Geometric Algebra \cite{Doran1993LG}, and Lie algebras as bivector algebras; here we offer a Clifford geometric construction of all the exceptional phenomena via their root systems. Of course, $H_4$ does not even have an associated Lie algebra and Lie group and therefore could not be constructed as a spin group or bivector algebra. We therefore advocate the root system as the most convenient, insightful and fundamental concept:  
for a root system, the quadratic form mentioned in the definition can always be used to enlarge the $n$-dimensional vector space $V$ to the corresponding $2^n$-dimensional Clifford algebra. The Clifford algebra  is therefore a very natural object to consider in this context, as its unified structure simplifies many problems both conceptually and computationally, rather than applying the linear structure of the space and the inner product separately. In particular, it provides the above (s)pinor double covers of the (special) orthogonal transformations, as well as geometric quantities that serve as unit imaginaries. We will see the benefits of those throughout the paper. We will largely be working with the Clifford algebra of 3D generated by three orthogonal unit vectors $e_1$, $e_2$ and $e_3$, which is itself an eight-dimensional vector space consisting of the elements
$$
  \underbrace{\{1\}}_{\text{1 scalar}} \,\,\ \,\,\,\underbrace{\{e_1, e_2, e_3\}}_{\text{3 vectors}} \,\,\, \,\,\, \underbrace{\{e_1e_2=Ie_3, e_2e_3=Ie_1, e_3e_1=Ie_2\}}_{\text{3 bivectors}} \,\,\, \,\,\, \underbrace{\{I\equiv e_1e_2e_3\}}_{\text{1 trivector}}.
$$

\section{The general spinor induction construction: $H_4$ as a group of rotations rather than reflections I, trinities and McKay correspondence}\label{sec_pin}

In this section we prove that any 3D root system yields a 4D root system via the spinor group obtained by multiplying together root vectors in the Clifford algebra  \cite{Dechant2012Induction}.
\begin{prop}[$O(4)$-structure of spinors]\label{HGA_O4}
The space of $\Cl(3)$-spinors $R=a_0+a_1e_2e_3+a_2e_3e_1+a_3e_1e_2$ can be endowed with a \emph{4D Euclidean norm} $|R|^2=R\tilde{R}=a_0^2+a_1^2+a_2^2+a_3^2$ induced by the  \emph{inner product} $(R_1,R_2)=\frac{1}{2}(R_1\tilde{R}_2+R_2\tilde{R}_1)$ between  two spinors $R_1$ and $R_2$. 
\end{prop}
This allows one to  reinterpret the group of 3D spinors generated from a 3D root system as a set of 4D vectors, which in fact can be shown to  satisfy the axioms of a root system as given in Definition \ref{DefRootSys}. 
\begin{thm}[Induction Theorem]\label{HGA_4Drootsys}
Any 3D root system gives rise to a spinor group $G$ which induces a root system in 4D.
\end{thm}
\begin{proof}
Check the two axioms for the root system $\Phi$ consisting of the set of 4D vectors given by the 3D spinor group:
\begin{enumerate}
\item By construction, $\Phi$ contains the negative of a root $R$ since spinors provide a double cover of rotations, i.e. if $R$ is in a spinor group $G$, then so is $-R$ , but no other scalar multiples (normalisation to unity). 
\item $\Phi$ is invariant under all reflections with respect to the inner product $(R_1,R_2)$ in Proposition \ref{HGA_O4} since $R_2'=R_2-2(R_1, R_2)/(R_1, {R}_1) R_1=-R_1\tilde{R}_2R_1\in G$ for $R_1, R_2 \in G$ by the closure property of the group $G$ (in particular $-R$ and $\tilde{R}$ are in $G$ if $R$ is). 
\end{enumerate}
\end{proof}

Since the irreducible 3D root systems are $(A_3, B_3, H_3)$, the induction construction also yields three induced root systems in 4D. These are in fact the exceptional root systems in 4D $(D_4, F_4, H_4)$. Both sets of three are in fact amongst Arnold's trinities following his observation that  $(\mathbb{R},\mathbb{C},\mathbb{H})$ \cite{Arnold1999symplectization,Arnold2000AMS} form a basic unit of three which can be extended to analogous sets of three such as the corresponding projective spaces, Lie algebras of $E$-type $(E_6, E_7, E_8)$, spheres, Hopf fibrations etc.
This is the first very tentative hint that the icosahedral group $H_3$ might be related to the exceptional geometry $E_8$. 
 Arnolds original link between $(A_3, B_3, H_3)$ and $(D_4, F_4, H_4)$  is similarly extremely convoluted/indirect, whilst our construction presents a novel  direct link between the two; we will first make the $H_3$ and $E_8$ connection more suggestive below, before making it explicit for the first time. 
These root systems are intimately linked to the Platonic solids \cite{Dechant2013Platonic}. There are 5 in three dimensions and 6 in four dimensions: $A_3$ describes the reflection symmetries of the tetrahedron, $B_3$  those of the cube and octahedron (which are dual under the exchange of midpoints of faces and vertices), and $H_3$ describes the symmetries of the dual pair icosahedron and dodecahedron (the rotational subgroup is denoted by $I= A_5$), whilst the 4D Coxeter groups describe the symmetries of the 4D Platonic solids. But this time the connection is much more immediate: the Platonic solids   are actually root systems themselves (or duals thereof). $D_4$ is the $24$-cell (self-dual), an analogue of the tetrahedron, which is also related to the $F_4$ root system. The $H_4$ root system is the Platonic solid the $600$-cell with its dual, the $120$-cell  (also Platonic), having the same symmetry.  $A_1^3$ generates  $A_1^4$, which constitutes the $16$-cell -- its dual is the $8$-cell, both are Platonic solids.  There is thus an abundance of root systems in 4D giving the Platonic solids, which is essentially due to the accidentalness of the spinor induction theorem, compared to arbitrary dimensions where the only root systems and Platonic solids are  $A_n$ ($n$-simplex), $B_n$  ($n$-hypercube and $n$-hyperoctahedron) and $D_n$. 
In fact the only 4D Platonic solid that is not equal or dual to a root system is the $5$-cell with symmetry group $A_4$, which of course could not be a root system, as its odd number ($5$) of vertices violates the first root system axiom. 

The induced root systems $(D_4, F_4, H_4)$ are precisely the exceptional ones in 4D, where we count $D_4$ as exceptional because of its exceptional triality symmetry (permutation symmetry of the three legs in the diagram accidental in 4D). It is of great importance in string theory, showing the equivalence of the Ramond-Neveu-Schwarz and the Green-Schwarz strings. $F_4$ is the only $F$-type root system, and $H_4$ is the largest non-crystallographic root system. They therefore naturally form a trinity. 

On top of  the exceptional nature of these root systems (their existence), they also have very unusual automorphism groups. This is readily shown via the above spinor construction: 
\begin{thm}[Spinorial symmetries]\label{HGAsymmetry}
A root system induced through the Clifford spinor construction via a binary polyhedral spinor group $G$ has an automorphism group that trivially contains two factors of the respective spinor group $G$ acting from the left and from the right.
\end{thm}
This systematises many case-by-case observations on the structure of the automorphism groups \cite{Koca2006F4,Koca2003A4B4F4} (for instance, the automorphism group of the $H_4$ root system is $2I\times 2I$ -- in the spinor picture, it is not surprising that $2I$ yields both the root system and the two factors in the automorphism group), and shows that all of the 4D geometry is already contained in 3D \cite{Dechant2012CoxGA} in the following sense. 
In terms of quaternionic representations of a 4D root system (e.g. $H_4$), the 3D root system ($H_3$) is often remarked in the literature to consist precisely of the pure quaternions, i.e. those without a real part, and the full 4D group can be generated from these under quaternion multiplication. This is very poorly understood in the literature: in our picture the spinors are just isomorphic to the quaternions whilst when the inversion  $\pm I$ is contained in the full group, one can trivially Hodge dualise vectors to bivectors (pure quaternions).  However, this statement is not true when the inversion is not contained in the group (e.g.  for $A_3$). Conversely, the spinorial induction construction still works perfectly well yielding $D_4$. I.e. the fact that the 3D group is generated by the subset of the pure quaternions amongst the 4D group under quaternion multiplication is just a corollary of spinor induction when the inversion is contained in the group. The literature has it `backwards' deriving the 3D root system from the 4D root system (under inconsistent conditions), whilst in our picture the 3D root systems are more fundamental as they allow us to construct the 4D root system without any further information. Moreover, the `quaternionic generators' generating the 4D groups via quaternion multiplication cited in the literature are easily seen to just be the products of pairs of 3D simple roots (spinors) $\alpha_1\alpha_2$ and $\alpha_2\alpha_3$. Thus the 4D group clearly does not contain anything that was not already contained in the 3D group, and we therefore argue that 3D root systems and spinor induction are fundamental rather than the 4D root systems. 

We consider the examples of the inductions of $A_1^3\rightarrow A_1^4$ and $H_3\rightarrow H_4$ in more detail. For the former, the $6$ roots can be chosen as $\pm e_1, \pm e_2, \pm e_3$. Pairwise products yield the $8$ spinors $\pm 1, \pm e_1e_2, \pm e_2e_3, \pm e_3e_1$. 3D spinors have four components ($1$, $e_1e_2$, $e_2e_3$, $e_3e_1$), such that these $8$ spinors give the $8$ vertices of the (4D root system and Platonic solid) $16$-cell 
$$(\pm 1, 0, 0, 0) \text{ (8 permutations).}$$
 We now also construct the spinor group generated by the simple roots of $H_3$, which we take as $$\alpha_1=e_2, \alpha_2=-\frac{1}{2}((\tau -1)e_1+e_2+\tau e_3),\text{ and }\alpha_3=e_3.$$ Under  multiplication with the geometric product they generate a group of $240$  pinors doubly covering the $120$ elements of $H_3$. Its even subgroup consists of $120$ spinors doubly covering the rotational subgroup $A_5$, e.g. $\alpha_1\alpha_2=-\frac{1}{2}(1-(\tau -1)e_1e_2+\tau e_2e_3)$ and $\alpha_2\alpha_3=-\frac{1}{2}(\tau-(\tau -1)e_3e_1+e_2e_3)$. Taking the components of these $120$ spinors in 4D yields exactly the  $H_4$ root system
$$(\pm 1, 0, 0, 0) \text{ (8 permutations) }$$
$$\frac{1}{2}(\pm 1, \pm 1, \pm 1, \pm 1) \text{ (16 permutations) }$$
$$\frac{1}{2}(0, \pm 1, \pm \sigma, \pm \tau) \text{ (96 even permutations) }.$$
 This is very surprising from a Coxeter perspective, as one usually thinks of $H_3$ as a subgroup of $H_4$, and therefore of $H_4$ as more fundamental. However,  we have seen that $H_4$ does not actually contain anything that is not already in $H_3$, which makes $H_3$ more fundamental \cite{Dechant2012CoxGA}. 
From a Clifford perspective it is not at all surprising to find this group of $120$ spinors (the binary icosahedral group $2I$)  since it is well-known that Clifford algebra provides a simple construction of the Spin groups. However, this is extremely surprising  from the conventional Coxeter and Lie group point of view. 

It is convenient to have all the four different types of polyhedral groups (chiral, full, binary, pin) in a unified framework within the Clifford algebra, rather than using $SO(3)$ matrices for the rotations and then having to use $SU(2)$ matrices for the binary groups, as one can perform all the different group operations with multivectors in the same Clifford algebra. For instance, the spinor group $2I$ consists of $120$ elements and $9$ conjugacy classes, as one can easily confirm by explicit computation. 
 $A_5$  has five conjugacy classes  and is of order $60$, which implies that it has five irreducible representations of dimensions  $1$, $3$, $\bar{3}$, $4$ and $5$. The nine conjugacy classes of the binary icosahedral group $2I$ of order $120$ therefore mean that it has a further four irreducible spinorial representations $2_s$, $2_s'$, ${4_s}$ and ${6_s}$. 
This binary icosahedral group has a mysterious two-fold connection with the affine Lie algebra $E_8^+$  via the so-called McKay correspondence \cite{Mckay1980graphs}, which in fact applies to all other binary polyhedral groups and the affine Lie algebras of $ADE$-type (c.f. Figure \ref{figMcKay}): 
 firstly, we can define a graph by assigning a node to each of the nine irreducible representation of the binary icosahedral group where we connect nodes according to its tensor product structure: each irreducible representation is represented by a node and it is only connected to other nodes corresponding to those irreducible representations that are contained in its tensor product with the irreducible representation $2_s$. For instance, tensoring the trivial representation $1$ with $2_s$ trivially gives $1\otimes 2_s=2_s$ and thus  $1$ is only connected to $2_s$; $2_s\otimes 2_s=1+3$, such that $2_s$ is connected to $1$ (as we know) and also to $3$, and so on. The graph that is built up in this way is precisely the Dynkin diagram of affine $E_8$ (Figure \ref{figE6aff}). Secondly the order of the Coxeter element  (the product of all the simple reflections $\alpha_1\dots \alpha_8$), the Coxeter number $h$, is $30$ for $E_8$, which is also exactly the sum of the dimensions of the irreducible representations of $2I$, $\sum d_i$. As already stated, these both extend to a correspondence between  all binary polyhedral groups and the $ADE$-type affine Lie algebras.

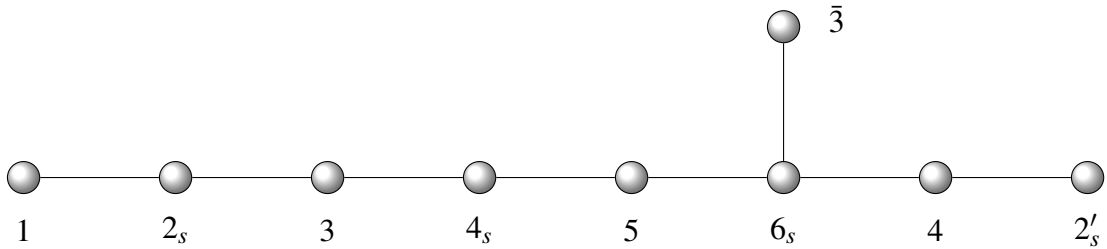
\begin{figure}
	\begin{center}
\begin{tikzpicture}[
knoten/.style={        circle,      inner sep=.15cm,        draw}
]
\node at (-1,.7) (knoten0) [knoten,   color=white!0!black, ball color=white ] {};
\node at  (1,.7) (knoten1) [knoten,   color=white!0!black, ball color=white ] {};
\node at  (3,.7) (knoten2) [knoten,   color=white!0!black, ball color=white ] {};
\node at  (5,.7) (knoten3) [knoten,   color=white!0!black, ball color=white ] {};
\node at  (7,.7) (knoten4) [knoten,   color=white!0!black, ball color=white ] {};
\node at  (9,.7) (knoten6) [knoten,   color=white!0!black, ball color=white ] {};
\node at  (11,.7) (knoten7) [knoten,   color=white!0!black, ball color=white ] {};
\node at  (13,.7) (knoten8) [knoten,   color=white!0!black, ball color=white ] {};
\node at  (9,2.7) (knoten9) [knoten,   color=white!0!black, ball color=white ] {};

\node at  (-1,0) (alpha0)  {$1$};
\node at  (1,0)  (alpha1) {$2_s$};
\node at  (3,0)  (alpha2) {$3$};
\node at  (5,0)  (alpha3) {$4_s$};
\node at  (7,0)  (alpha4) {$5$};
\node at  (9,0)  (alpha5) {$6_s$};
\node at  (11,0)  (alpha6) {$4$};

\node at (9.7,2.8)  (alpha7) {$\bar{3}$};
\node at (13,0) (alpha8) {$2_s'$};

\path  (knoten0) edge (knoten1);
\path  (knoten1) edge (knoten2);
\path  (knoten2) edge (knoten3);
\path  (knoten3) edge (knoten4);
\path  (knoten4) edge (knoten6);
\path  (knoten6) edge (knoten9);
\path  (knoten6) edge (knoten7);
\path  (knoten7) edge (knoten8);

\end{tikzpicture} 
\end{center}
\caption[$E_6^+$]{The graph depicting the tensor product structure of the binary icosahedral group $2I$ is the same as the Dynkin diagram of  affine  $E_8$. }
\label{figE6aff}
\end{figure}

We can therefore extend the connection between  $(A_3, B_3, H_3)$ and $(2T, 2O, 2I)$ via the McKay correspondence which links  $(2T, 2O, 2I)$ and $(E_6, E_7, E_8)$ to a new connection between $(A_3, B_3, H_3)$ and $(E_6, E_7, E_8)$ via Clifford spinors, which does not seem to be known. In particular, $(12, 18, 30)$ is one of the two connections in the McKay correspondence denoting the Coxeter number of the affine Lie algebra as well as the sum of the dimensions of the irreducible representations of the binary polyhedral group. 
We note that $(12, 18, 30)$ is more fundamentally the number of roots $\Phi$ in the 3D root systems $(A_3, B_3, H_3)$, which therefore feeds all the way through to the binary polyhedral groups and via the McKay correspondence to the affine Lie algebras. Our construction therefore makes deep connections between trinities, and puts the McKay correspondence into a wider framework, as shown in Figure \ref{figMcKay}. It is worth noting that the affine Lie algebra and the 4D root system trinities have identical Dynkin diagram symmetries: $D_4$ and $E_6^+$ have triality $S_3$, $F_4$ and $E_7^+$ have an $S_2$ symmetry and $H_4$ and $E_8^+$ only have $S_1$, but are intimately related as explained in Section \ref{sec_spin}.  This therefore again suggests that there is a link between the icosahedron and $E_8$, which is slightly more explicit than just from observing the trinities. In the next section we will show a new, completely explicit direct connection within the Clifford algebra of 3D by identifying the $240$ roots of $E_8$ with the $240$ pinors doubly covering the elements of the icosahedral group $H_4$ (right of Figure \ref{figMcKay}) \cite{Dechant2016Birth}.

\begin{figure}
	\begin{center}

\includegraphics[width=16cm]{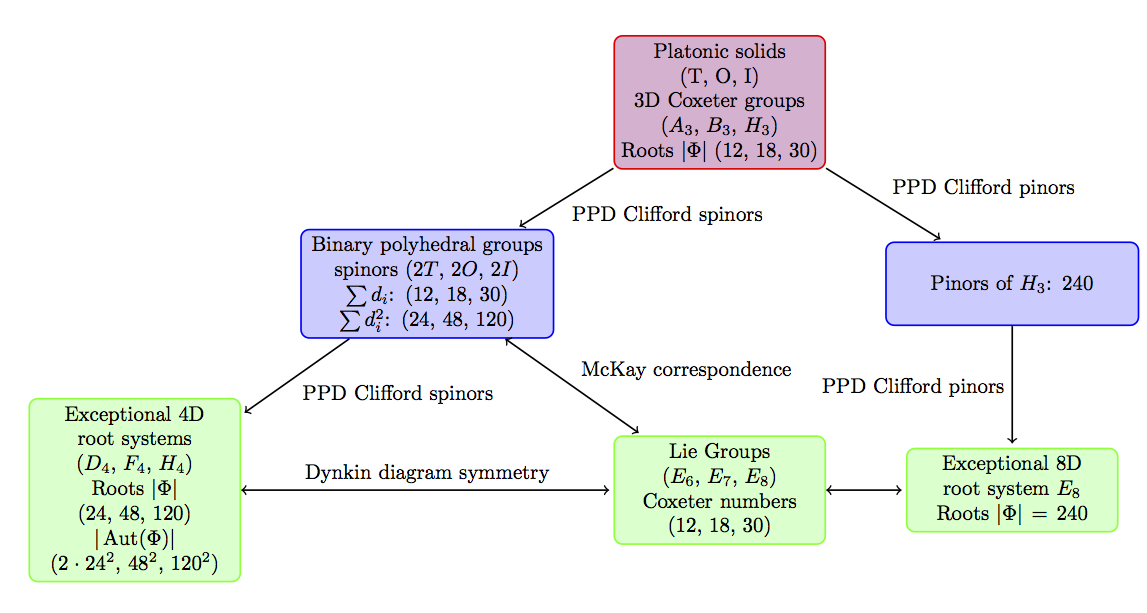}

\end{center}
\caption[$E_6^+$]{Web of connections putting the original trinities and McKay correspondence into a much wider context. The connection between the sum of the dimensions of the irreducible representations $d_i$ of the binary  polyhedral groups and the Coxeter number of the Lie algebras actually goes all the way back to the number of roots in the 3D root systems $(12, 18, 30)$. These then induce the binary polyhedral groups (linked via the  McKay correspondence to the $E$-type affine Lie algebras) and the 4D root systems via the Clifford spinor construction. The new pinor construction links $H_3$ directly with $E_8$ explaining the latter entirely within 3D geometry.}
\label{figMcKay}
\end{figure}

\section{The birth of $E_8$ out of the spinors of the icosahedron}\label{birth}

Previously, we have constructed the $120$ elements of $2I$, which can be reinterpreted as the $120$ roots of $H_4$
$$(\pm 1, 0, 0, 0) \text{ (8 permutations) }$$
$$\frac{1}{2}(\pm 1, \pm 1, \pm 1, \pm 1) \text{ (16 permutations) }$$
$$\frac{1}{2}(0, \pm 1, \pm \sigma, \pm \tau) \text{ (96 even permutations)}.$$
A set of simple roots for these is given e.g. by $a_1=\frac{1}{2}(-\sigma, -\tau, 0, -1)$, $a_2=\frac{1}{2}(0, -\sigma, -\tau,  1)$, $a_3=\frac{1}{2}(0,1, -\sigma, -\tau)$ and $a_4=\frac{1}{2}(0, -1, -\sigma, \tau)$.

As we have discussed above the $H_3$ root system contains  the inversion $\pm e_1e_2e_3=\pm I$. It simply Hodge dualises  the $120$ spinors doubly covering the even  subgroup $A_5$ of $60$ rotations to create a second copy  consisting of vector and pseudoscalar parts. This gives the $120$ additional group elements necessary to provide the $240$  pinor double covering of the group $H_3=A_5\times \mathbb{Z}_2$ of order $120$. These therefore consist of a copy of the $120$ spinors (essentially $H_4$) and another copy multiplied by $I$. These are now valued in the full  8D Clifford algebra of 3D space consisting of scalar, vector, bivector and trivector parts.
For this set of $240$ pinors in the 8D Clifford algebra of 3D we now define a `reduced inner product': we keep the spinor copy of $H_4$ and multiply the copy $IH_4$ by $\tau I$. We then take inner products taking into account the recursion relation $\tau^2=\tau+1$ but then  rather  than taking the usual inner product consisting of both the integer with the $\mathbb{Z}[\tau]$-integer part $(\cdot, \cdot) = a+\tau b$ we extract only the integer part by defining the reduced inner product \cite{Wilson1986E8, MoodyPatera:1993b} $$(\cdot, \cdot)_\tau=(a+\tau b)_\tau:=a.$$ 
This set of $240$  includes the above choice of simple roots of $H_4$ 
$$\alpha_1:=a_1=\frac{1}{2}(-\sigma, -\tau, 0, -1),$$ 
$$\alpha_2:=a_2=\frac{1}{2}(0, -\sigma, -\tau,  1),$$ 
$$\alpha_3:=a_3=\frac{1}{2}(0,1, -\sigma, -\tau) \text { and }$$  
$$\alpha_8:=a_4=\frac{1}{2}(0, -1, -\sigma, \tau)$$
 along with their $\tau$-multiples 
$$\alpha_7:=\tau a_1=\frac{1}{2}(1, -\tau-1, 0, -\tau),$$ 
$$\alpha_6:=\tau a_2=\frac{1}{2}(0, 1, -\tau-1,  \tau),$$
$$\alpha_5:=\tau a_3=\frac{1}{2}(0,\tau, 1, -\tau-1) \text{ and }$$ 
$$\alpha_4:=\tau a_4=\frac{1}{2}(0, -\tau, 1, \tau+1).$$ 
This choice of inner product combines the two sets of $H_4$  into the $E_8$ diagram by changing the links. 
The link labelled by $5$ in the $H_4$ diagram
$$(a_3,a_4)=\frac{1}{4}(-1+\sigma^2-\tau^2)=\frac{1}{4}(-1+\sigma+1-\tau-1)=\frac{1}{4}(-1+1-\tau-\tau)=-\frac{\tau}{2},$$ means that these roots are instead orthogonal with respect to the reduced inner product
$$(\alpha_3,\alpha_8)_\tau=(a_3,a_4)_\tau=\left(-\frac{\tau}{2}\right)_\tau=0.$$
 Likewise,
$$(\alpha_5,\alpha_4)_\tau=(\tau \alpha_3,\tau \alpha_4)_\tau=\left(-\frac{\tau^3}{2}\right)_\tau=\left(-\frac{2\tau+1}{2}\right)_\tau=-\frac{1}{2},$$
means that the $5$-labelled link from the other $H_4$ diagram gets turned into a simple link. 
There are additional (simple) links introduced from
$$(\alpha_5,\alpha_8)_\tau=(\tau \alpha_3,\alpha_8)_\tau=(\alpha_3,\tau \alpha_8)_\tau=(\alpha_3,\alpha_4)_\tau=-\left(\frac{\tau^2}{2}\right)_\tau=-\left(\frac{\tau+1}{2}\right)_\tau=-\frac{1}{2},$$
whilst the simply-connected nodes of the original $H_4$ diagrams are not affected 
$$(\alpha_1,\alpha_2)_\tau=\left(-\frac{1}{2}\right)_\tau=-\frac{1}{2}.$$

This set of simple roots yields the following Cartan matrix 
$$ \begin{pmatrix}
   2&-1&0&0&0&0&0&0  
\\ -1&2&-1&0&0&0&0  &0
\\ 0&-1&2&-1&0&0&0&0
\\ 0&0&-1&2&-1&0&0&0  
\\ 0&0&0&-1&2&-1 &0&-1 
\\ 0&0&0&0&-1&2 &-1 &0
\\ 0&0&0&0&0&-1&2 &0
\\ 0&0&0&0&-1&0&0 &2\end{pmatrix},$$
which is just the $E_8$ Cartan matrix \cite{Dechant2016Birth}. The $240$ icosahedral pinors therefore  give the $240$ roots of $E_8$ under the reduced inner product and it is straightforward (if tedious) to check closure under reflections with respect to the reduced inner product. 

Surprisingly, the $E_8$ root system has therefore been hidden in plain sight within the  geometry of the Platonic icosahedron for three millennia, without anyone noticing. As with the 4D induction construction, this discovery seems only possible in Clifford algebra -- there is much prejudice against the usefulness of Clifford algebras since they have matrix representations and usually matrix methods are equivalent if less insightful -- but the 4D and 8D induction constructions are to my knowledge the only results that {required} Clifford algebra and were hitherto invisible to standard matrix methods.

\section{$H_4$ as a group of rotations rather than reflections II:  from $E_8$}\label{sec_spin}

\begin{figure}
	\begin{center}
\begin{tikzpicture}[
knoten/.style={        circle,      inner sep=.15cm,        draw}
]
\node at  (1,.7) (knoten1) [knoten,  color=white!0!black, ball color=white ] {};
\node at  (3,.7) (knoten2) [knoten,  color=white!0!black, ball color=black] {};
\node at  (5,.7) (knoten3) [knoten,  color=white!0!black, ball color=red ] {};
\node at  (7,.7) (knoten4) [knoten,  color=white!0!black, ball color=green] {};

\node at  (9,.7) (knoten5) [knoten,  color=white!0!black, ball color=red] {};
\node at (11,.7) (knoten6) [knoten,  color=white!0!black, ball color=black] {};
\node at (13,.7) (knoten7) [knoten,  color=white!0!black, ball color=white] {};
\node at (9,2.2) (knoten8) [knoten,  color=white!0!black, ball color=green] {};

\node at  (1,0)  (alpha1) {$\alpha_1$};
\node at  (3,0)  (alpha2) {$\alpha_2$};
\node at  (5,0)  (alpha3) {$\alpha_3$};
\node at  (7,0)  (alpha4) {$\alpha_4$};
\node at  (9,0)  (alpha7) {$\alpha_5$};
\node at (11,0)  (alpha6) {$\alpha_6$};
\node at (13,0)  (alpha5) {$\alpha_7$};
\node at (9,2.8) (alpha8) {$\alpha_8$};

\path  (knoten1) edge (knoten2);
\path  (knoten2) edge (knoten3);
\path  (knoten3) edge (knoten4);
\path  (knoten4) edge (knoten5);
\path  (knoten5) edge (knoten6);
\path  (knoten6) edge (knoten7);
\path  (knoten5) edge (knoten8);

\end{tikzpicture}

	\begin{tikzpicture}[
	    knoten/.style={        circle,      inner sep=.15cm,        draw}  
	   ]

	  \node at (1,1.5) (knoten1) [knoten,  color=white!0!black, ball color=white ] {};
	  \node at (3,1.5) (knoten2) [knoten,  color=white!0!black, ball color=black ] {};
	  \node at (5,1.5) (knoten3) [knoten,  color=white!0!black, ball color=red ] {};
	  \node at (7,1.5) (knoten4) [knoten,  color=white!0!black, ball color=green ] {};

	\node at (1,2)  (a1) {$a_1$};
	\node at (3,2)  (a2) {$a_2$};
		\node at (5,2)  (a3) {$a_3$};
	\node at (6,1.75)  (tau) {$5$};
\node at (4,2.5)   (tau1) {};
	\node at (7,2)  (a4) {$a_4$};

	  \path  (knoten1) edge (knoten2);
	  \path  (knoten2) edge (knoten3);
	  \path  (knoten3) edge (knoten4);

	\end{tikzpicture}
	
  \caption[$E_8$]{Coxeter-Dynkin diagram folding and projection from $E_8$ to $H_4$: one defines the new generators ${s_{a_1}=s_{\alpha_1} s_{\alpha_7}},\,\,\,s_{a_2}=s_{\alpha_2} s_{\alpha_6},\,\,\,{s_{a_3}=s_{\alpha_3} s_{\alpha_5}},\,\,\,{s_{a_4}=s_{\alpha_4} s_{\alpha_8}}$, which themselves satisfy ${H_4}$ relations. }
\label{figE8}
\end{center}
\end{figure}
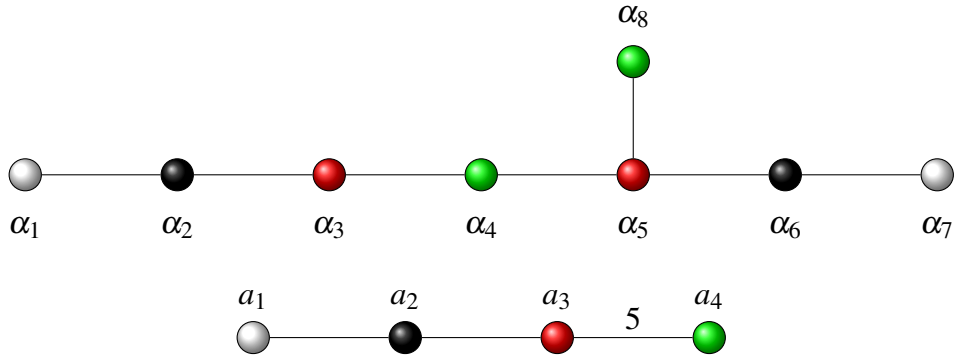

The usual view is the reverse of the process shown in the previous section, inducing $E_8$ from $H_3$, via two intermediate copies of $H_4$.
$E_8$ has an $H_4$ subgroup, as can be shown via Coxeter-Dynkin diagram foldings  \cite{Shcherbak:1988}:

We consider now the Clifford algebra in 8D with the usual Euclidean metric and take the simple roots $\alpha_1$ to  $\alpha_8$ of $E_8$ as shown in  Fig. \ref{figE8}. The simple reflections corresponding to the simple roots are thus just given via  
 $s_\alpha v=-\alpha v \alpha$. The Coxeter element $w$ is defined as the product of all these eight simple reflections, and in Clifford algebra it is therefore simply given by the corresponding Coxeter versor $W=\alpha_1\dots \alpha_8$ acting via sandwiching as $w\lambda=\tilde{W}\lambda W$. Its order, the Coxeter number $h$ (i.e. $W^h=\pm 1$), is $30$ for $E_8$ as mentioned previously. 

As illustrated in Figure \ref{figE8}, one can now define certain combinations of pairs of reflections (according to a four-fold colouring of the diagram, or corresponding to roots on top of each other in a Dynkin diagram folding), e.g. $s_{a_1}=s_{\alpha_1} s_{\alpha_7}$ etc. In a Clifford algebra setup these are just given by the products of root vectors ${a_1}={\alpha_1} {\alpha_7}$,   ${a_2}={\alpha_2} {\alpha_6}$, ${a_3}={\alpha_3} {\alpha_5}$ and   ${a_4}={\alpha_4} {\alpha_8}$ (this is essentially a partial folding of the usual alternating folding/two-colouring used in the construction of the Coxeter plane with symmetry group $I_2(h)$, see the next section). It is easy to show that the subgroup with the generators $s_{a_i}$  in fact satisfies the relations of the  $H_4$ Coxeter group \cite{Bourbaki1981Lie, Shcherbak:1988}: because of the Coxeter relations for $E_8$ and the orthogonality of the combined pair the combinations $s_a$ are easily seen to be involutions, and the 3-fold relations are similarly obvious from the  Coxeter relations; only for the 5-fold relation does one have to perform an explicit calculation in terms of the reflections with respect to the root vectors. This is thus particularly easy by multiplying together vectors in the Clifford algebra, rather than by concatenating two reflection formulas of the type shown in Definition \ref{DefRootSys} -- despite it only consisting of two terms, concatenation gets convoluted quickly, unlike multiplying together multivectors. 

Since the combinations $s_a$ are pairs of reflections, they are obviously rotations in the eight-dimensional space, so this $H_4$ group acts as rotations in the full space, but as a reflection group in a 4D subspace. The $H_4$ Coxeter element is given by multiplying together the four combinations $a_i$ -- its Coxeter versor is therefore trivially seen to be the same as that of $E_8$ (up to sign, since orthogonal vectors anticommute) and the Coxeter number of $H_4$ is thus the same as that of $E_8$, $30$. The projection of the $E_8$ root system onto the Coxeter plane consists of two copies of the projection of $H_4$ into the Coxeter plane, with a relative factor of $\tau$ (see the next section and in particular Figure \ref{figCoxPlE8}). 
This is also related to the fact that on the level of the root system there is a projection which maps the $240$ roots of $E_8$ onto the $120$ roots of $H_4$ and their $\tau$-multiples in one of the $H_4$-invariant 4-subspaces \cite{MoodyPatera:1993b, DechantTwarockBoehm2011E8A4} (c.f. previous section). This is essentially the exact reverse of finding $E_8$ from the two copies $H_4+\tau H_4$ in the last section. 
We therefore now consider the Coxeter plane itself. 

\section{The Coxeter plane}\label{Coxeterplane}

In this section, we consider the geometry of the Coxeter plane e.g. achieved by a complete folding/two-fold colouring of the $E_8$ Dynkin diagram (Figure \ref{figCoxPlPF}). 
For a given Coxeter element $w$ of any root system, there is a unique  plane called the Coxeter plane on which  $w$ acts as a rotation by $2\pi/h$.
Projection of a root system onto the Coxeter plane is thus a  way of visualising any finite Coxeter group, for instance the well-known representation of $E_8$ is such a projection of the $240$ vertices of the root system onto the Coxeter plane. 
In the standard theory there is an unnecessary complexification of the real geometry followed by taking real sections again just so that complex eigenvalues $\exp(2\pi i m/h)$  of $w$, for some integers $m$ called exponents, can be found \cite{Humphreys1990Coxeter}. 
Unsurprisingly, in Clifford algebra these complex structures arise naturally, and the complex `eigenvectors' are in fact eigenplanes where the Coxeter element acts as a rotation. We therefore systematically factorise Coxeter versors of root systems in the Clifford algebra, which gives the eigenplanes and exponents algebraically. We briefly discuss the 2D case of the two-dimensional family  of non-crystallographic Coxeter groups $I_2(n)$, followed by the three-dimensional groups $A_3$, $B_3$ and $H_3$ \cite{Dechant2012AGACSE}, before discussing the higher-dimensional examples.

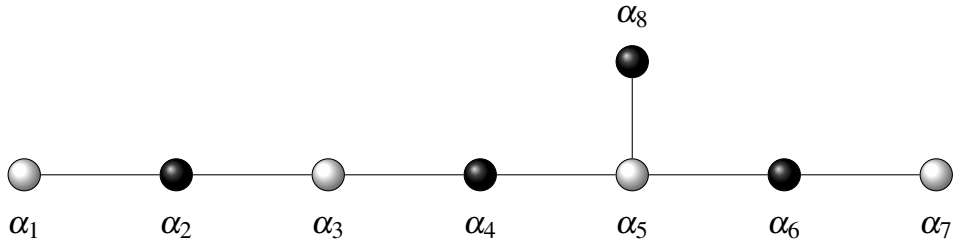
\begin{figure}
	\begin{center}
		\begin{tikzpicture}[
		knoten/.style={        circle,      inner sep=0.15cm,        draw}
		]
		\node at  (1,.7) (knoten1) [knoten,  color=white!0!black, ball color=white ] {};
		\node at  (3,.7) (knoten2) [knoten,  color=white!0!black, ball color=black] {};
		\node at  (5,.7) (knoten3) [knoten,  color=white!0!black, ball color=white ] {};
		\node at  (7,.7) (knoten4) [knoten,  color=white!0!black, ball color=black] {};

		\node at  (9,.7) (knoten5) [knoten,  color=white!0!black, ball color=white] {};
		\node at (11,.7) (knoten6) [knoten,  color=white!0!black, ball color=black] {};
		\node at (13,.7) (knoten7) [knoten,  color=white!0!black, ball color=white] {};
		\node at (9,2.2) (knoten8) [knoten,  color=white!0!black, ball color=black] {};

		\node at  (1,0)  (alpha1) {$\alpha_1$};
		\node at  (3,0)  (alpha2) {$\alpha_2$};
		\node at  (5,0)  (alpha3) {$\alpha_3$};
		\node at  (7,0)  (alpha4) {$\alpha_4$};
		\node at  (9,0)  (alpha7) {$\alpha_5$};
		\node at (11,0)  (alpha6) {$\alpha_6$};
		\node at (13,0)  (alpha5) {$\alpha_7$};
		\node at (9,2.8) (alpha8) {$\alpha_8$};

		\path  (knoten1) edge (knoten2);
		\path  (knoten2) edge (knoten3);
		\path  (knoten3) edge (knoten4);
		\path  (knoten4) edge (knoten5);
		\path  (knoten5) edge (knoten6);
		\path  (knoten6) edge (knoten7);
		\path  (knoten5) edge (knoten8);

		\end{tikzpicture}
	\end{center}

	\caption[dummy1]{Illustration of the geometry of the Coxeter plane via diagram foldings. Since any finite Coxeter group has a tree-like diagram, one can partition the simple roots into two  sets (black and white) of roots that are orthogonal within each set. Since the Cartan matrix is positive definite, it has a Perron-Frobenius eigenvector with all positive entries. This allows one to show the existence of the invariant Coxeter plane (by construction) as the bivector defined by the outer product of two vectors that are linear combinations of all the reciprocals of the white (respectively black) simple roots with the corresponding coefficients given by the entries in the Perron-Frobenius eigenvector.    }
	\label{figCoxPlPF}
	\end{figure}

The simple roots for $I_2(n)$  can  be taken as
 $\alpha_1=e_1$, $\alpha_2=-\cos{\frac{\pi}{n}}e_1+\sin{\frac{\pi}{n}}e_2$, which yields the Coxeter versor $W$ describing the $n$-fold rotation encoded by the $I_2(n)$ Coxeter element via 	$v\rightarrow wv=\tilde{W}vW$
as
\begin{equation}
W=\alpha_1\alpha_2=-\cos{\frac{\pi}{n}}+\sin{\frac{\pi}{n}}e_1e_2=-\exp{\left(-{\pi e_1e_2/n}\right)}.
\end{equation}
In Clifford algebra it is therefore immediately obvious that the action of the $I_2(n)$ Coxeter element is described by a versor that encodes rotations in the $e_1e_2$-Coxeter-plane and  yields $h=n$ since trivially $W^n=(-1)^{n+1}$. Since $I=e_1e_2$ is the bivector defining the plane of $e_1$ and $e_2$, it anticommutes with both $e_1$ and $e_2$ such that one can take $W$ through to the left (which reverses the bivector part) to arrive at the complex eigenvector equation 
$$
	v\rightarrow wv=\tilde{W}vW=\tilde{W}^2v=\exp{\left(\pm{2\pi I/n}\right)}v.
$$%
This yields the standard result for the complex eigenvalues,  however, in Clifford algebra it is now obvious that the complex structure is in fact given by the bivector describing the Coxeter plane (which is trivial for $I_2(n)$), and that the standard complexification is both unmotivated and unnecessary. 

More generally, if $v$ lies in a plane in which $W$ acts as a rotation, then 
$$	v\rightarrow wv=\tilde{W}vW=\tilde{W}^2v
$$%
still holds, whereas if $v$ is orthogonal such that the bivector describing the planes commutes with $v$, one just has
$$	v\rightarrow wv=\tilde{W}vW=\tilde{W}Wv=v.
$$%
Thus, if  $W$ factorises into orthogonal eigenspaces $W=W_1\dots W_k$  with $v$ lying in the plane defined by $W_1$, then all the orthogonal $W_i$s commute through and cancel out, whilst the one that defines the eigenplane that $v$ lies in described by the bivector $B_i$ gives the standard complex eigenvalue equation
$$	v\rightarrow wv=\tilde{W}vW=\tilde{W_1}\dots\tilde{W_k}vW_1\dots W_k=\tilde{W_1}^2\dots\tilde{W_k}W_kv=\tilde{W_1}^2v=\exp(2\pi B_im/h)v.
$$%
If $m$ is an exponent then so is $h-m$ since $w^{-1}$ will act as ${W_1}^2v=\exp(-2\pi B_im/h)v=\exp(2\pi B_i(h-m)/h)v$; in particular $1$ and $h-1$ are always exponents (from the Coxeter plane) -- in Clifford algebra it is easy to see that these are just righthanded and lefthanded rotations in the respective eigenplanes, with bivectors giving the complex structures. If $W$ has pure vector factors then these act as reflections and trivially yield the exponents $h/2$. 

The pin group/eigenblade description in Geometric Algebra therefore yields a wealth of novel geometric insight, and we now consider higher-dimensional examples. For 3D and 4D groups, the geometry is completely governed by the above 2D geometry in the Coxeter plane, since the remaining normal vector (3D) or bivector (4D) are trivially fixed.   For $H_3$ one has $h=10$ and  complex eigenvalues $\exp(2\pi mi/h)$ with the exponents $m=\lbrace 1, 5, 9\rbrace$. For simple roots $\alpha_1=e_2$, $-2\alpha_2=(\tau-1)e_1+e_2+\tau e_3$ and $\alpha_3=e_3$,  the Coxeter plane bivector is $B_C=e_1e_2+\tau e_3e_1$ and the Coxeter  versor $2W=-\tau e_2-e_3+(\tau-1)I$ (here $I=e_1e_2e_3$) with eigenvalues $\exp{\left(\pm{2\pi B_C/h}\right)}$, which corresponds to $m=1$ and $m=9$. 
The vector $b_C=B_C I=-\tau e_2-e_3$ orthogonal to the Coxeter plane can only get reversed (since the Coxeter element in 3D is an odd operation), so one has $-\tilde{W}b_CW=-b_C=\exp{\left(\pm{5\cdot 2\pi B_C/h}\right)}b_C$ which gives $m=5$. $A_3$ and $B_3$ are very similar, they have Coxeter numbers $h=4$ and $h=6$, respectively, and thus exponents $m=\lbrace 1, 2, 3\rbrace$ and $m=\lbrace 1, 3, 5\rbrace$. The exponents $1$ and $h-1$ correspond to a rotation in the Coxeter plane in which the Coxeter element acts by $h$-fold rotation, whilst the normal to the Coxeter plane gets simply reflected, corresponding to the cases $h/2$ ($m=2$ and $m=3$ for $A_3$ and $B_3$, respectively).

\begin{figure}
	\begin{center}
	 		      \begin{tabular}{@{}c@{ }c@{ }c@{ }}
						\begin{tikzpicture}
						\node (img) [inner sep=0pt,above right]
					{\includegraphics[width=3.5cm]{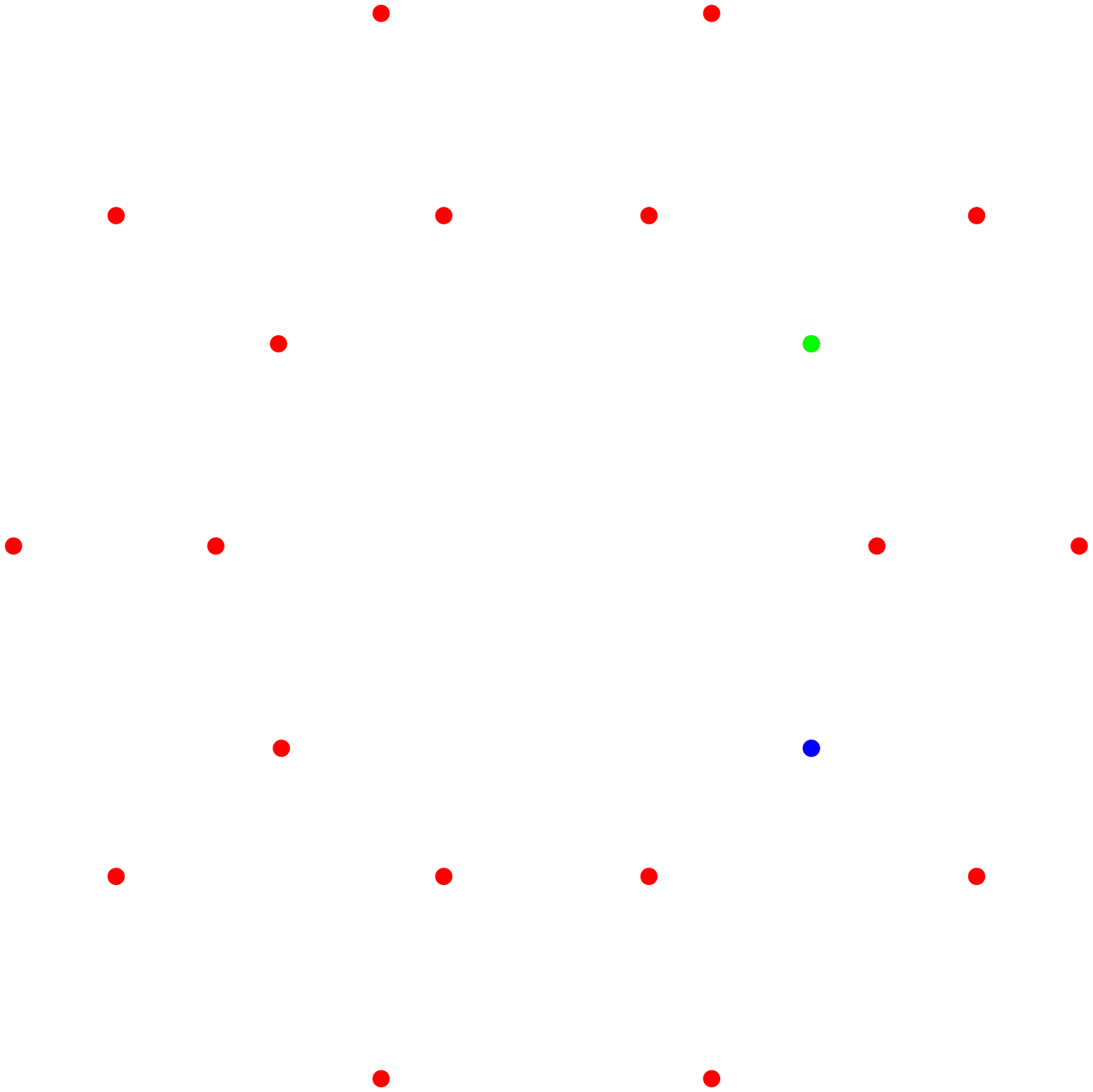}};
						\end{tikzpicture}&\hspace{1.5cm}
						\begin{tikzpicture}
							\node (img) [inner sep=0pt,above right]
						{\includegraphics[width=3.5cm]{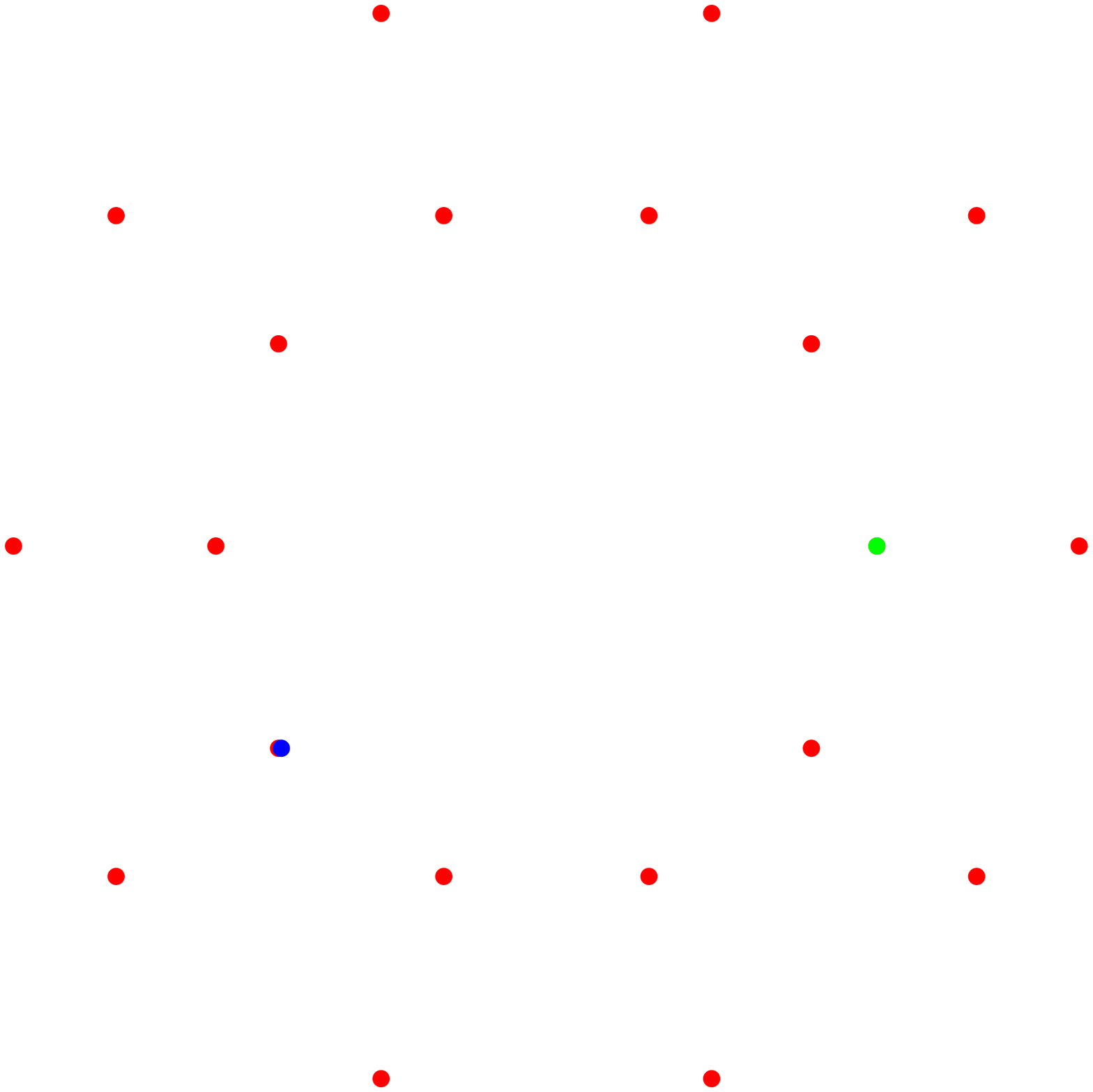}};
							\end{tikzpicture}&\hspace{1.5cm}	\begin{tikzpicture}
							\node (img) [inner sep=0pt,above right]		 				
							{\includegraphics[width=3.5cm]{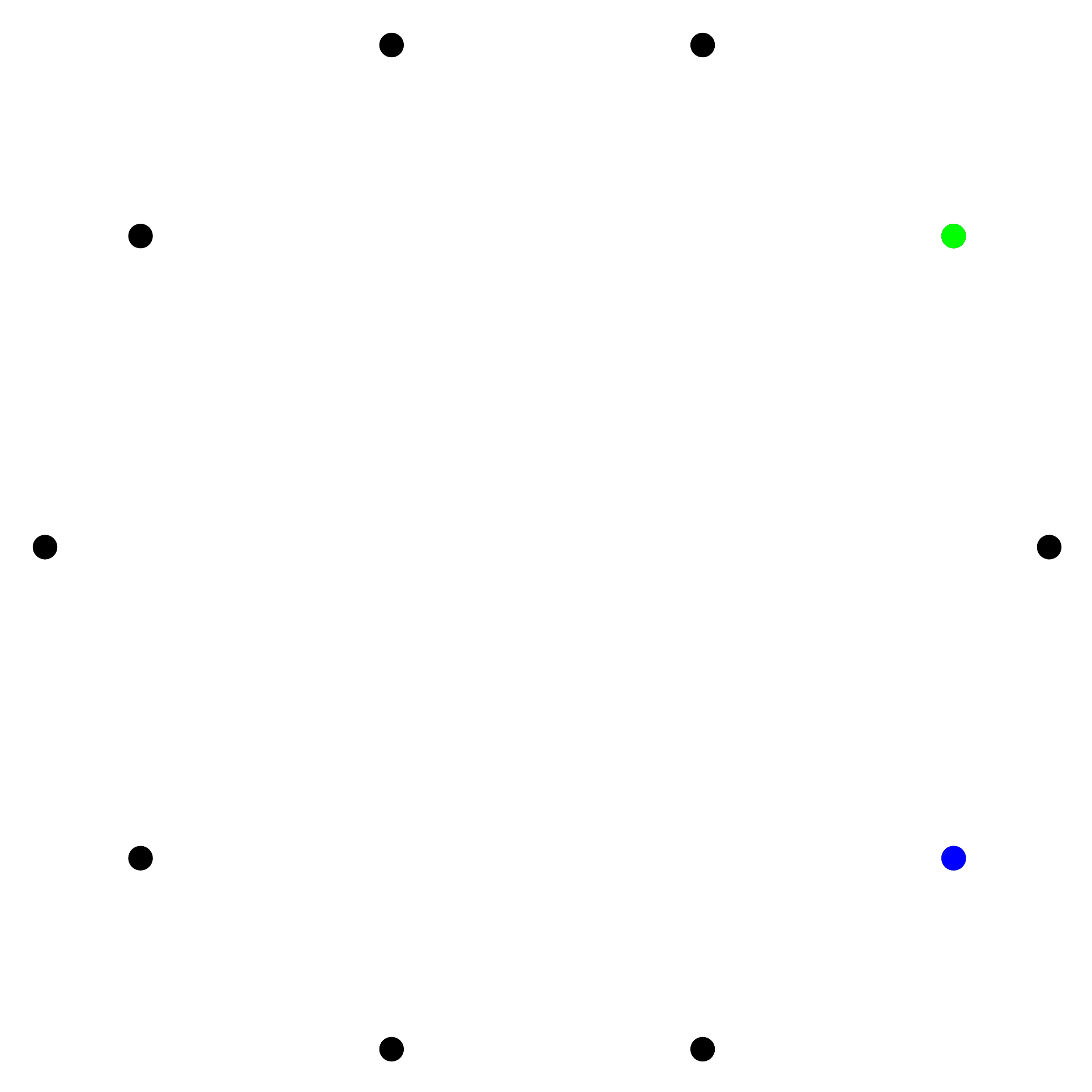}};
										\end{tikzpicture}
				\vspace{0.25cm}
			\\
					(a)&(b)&	(c) \\

						\begin{tikzpicture}
						\node (img) [inner sep=0pt,above right]		 				
						{\includegraphics[width=3.5cm]{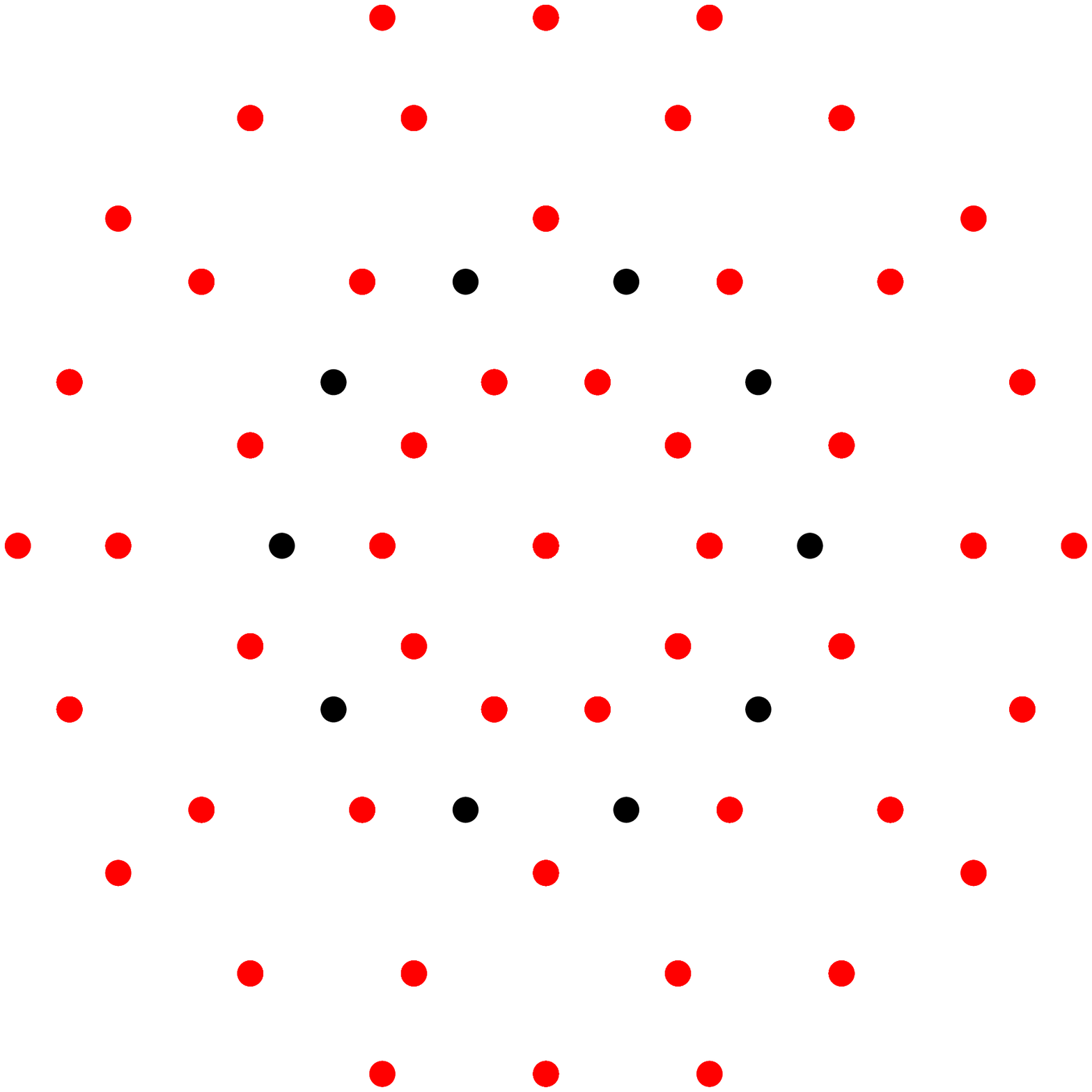}};
							\end{tikzpicture}&\hspace{1.5cm}
								\begin{tikzpicture}
							\node (img) [inner sep=0pt,above right]
									{\includegraphics[width=3.5cm]{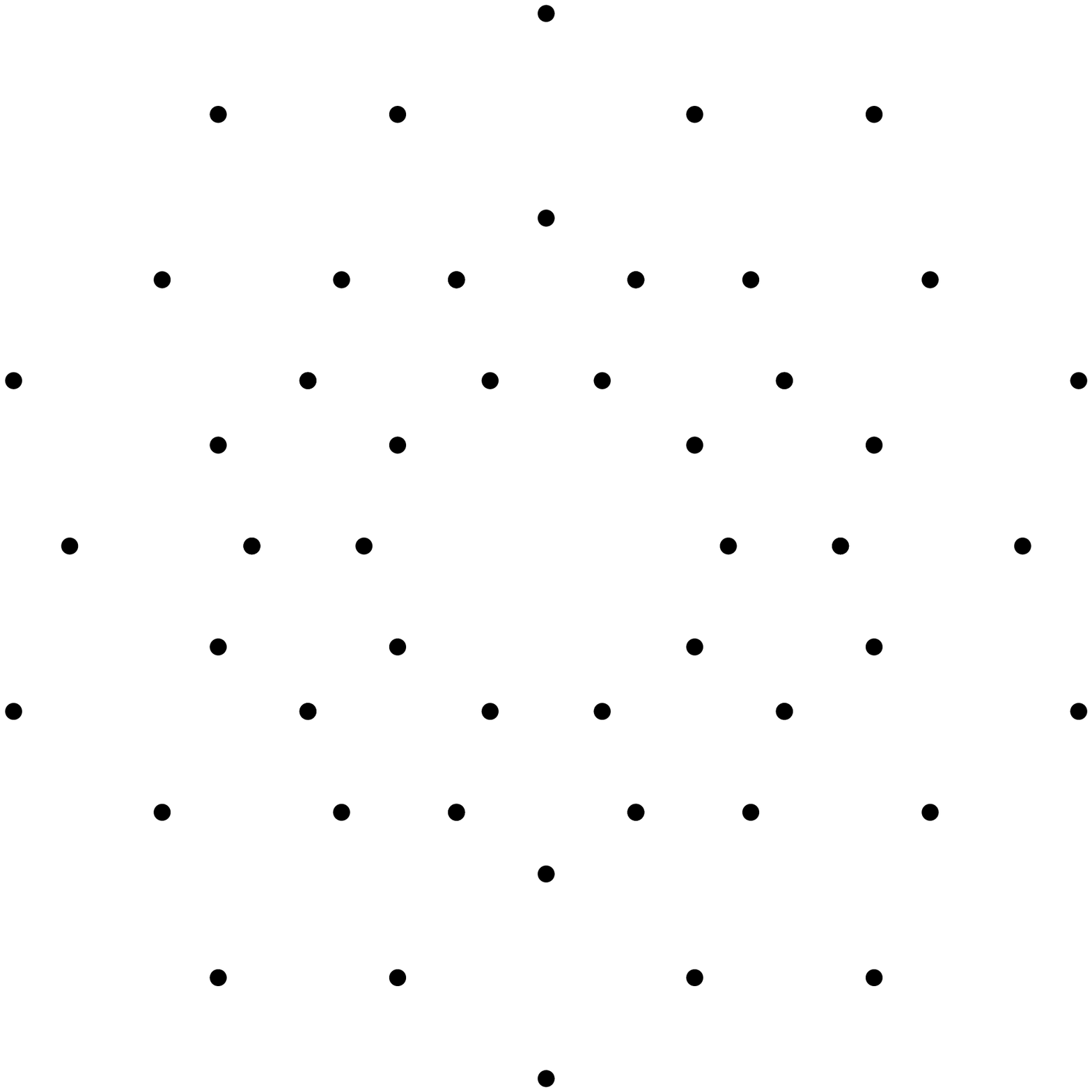}};
										\end{tikzpicture}&
								\vspace{0.25cm}
							\\
									(d)&(e) &\\
		\end{tabular}	
	
\end{center}
\caption[dummy1]{Coxeter projections of $A_4$ (a) and b)) and $H_2$ (panel c)), $H_2^{aff}$ (d)) and $D_6$ Coxeter projection (e)). The action of the Coxeter element as an $h$-fold rotation is visualised by the two coloured dots (green and blue) that are rotated into another by this rotation.  }
\label{figCoxPlA4}
\end{figure}

We now consider the four-dimensional cases  $A_4$, $B_4$, $D_4$, $F_4$ and $H_4$.
We explain the case of $A_4$ in detail, which is known to have exponents $\{1,2,3,4\}$.  
We take as the simple roots $\alpha_1=\frac{1}{\sqrt{2}}(e_2-e_1)$, $\alpha_2=\frac{1}{\sqrt{2}}(e_3-e_2)$, $\alpha_3=\frac{1}{\sqrt{2}}(e_4-e_3)$ and  $\alpha_4=\frac{1}{{2}}(\tau e_1+\tau e_2 +\tau e_3 +(\tau-2)e_4)$. It is easy to calculate that this choice of simple roots yields the correct $A_4$ Cartan matrix  and that reflections in them yield a root polytope of $20$ vertices. This matrix has Perron-Frobenius eigenvector $(1, \tau, \tau, 1)^T$ and  one can construct the Coxeter plane bivector as $B_C\propto -e_1e_3-e_1e_4+e_2e_3+e_2e_4-1/2(\tau-1)e_3e_4$ via the two vectors $e_3+e_4$ and $-e_1+e_2+e_3+(2\tau+1)e_4$ arrived at from the Perron-Frobenius eigenvector and the reciprocal frame of the simple roots as illustrated in Figure \ref{figCoxPlPF}. The Coxeter versor $W=\alpha_3\alpha_1\alpha_2\alpha_4$ is calculated to be $4W= 1-e_2e_3+e_1e_4+(\tau-1)(e_3e_4+e_2e_4-e_1e_3)-(\tau+1)e_1e_2-(2\tau-1)e_1e_2e_3e_4$. It is easy to show that $\tilde{W}B_CW=B_C$ and the Coxeter element therefore indeed stabilises the Coxeter plane. However, we are claiming that the Coxeter element can actually be written as bivector exponentials in the planes defined by $B_C$ and  $IB_C$, with angles given by the exponents  $\{1,2,3,4\}$. These are given as left- and righthanded rotations in the two planes as shown in Figure \ref{figCoxPlA4} a) and b). The Coxeter projection of the $20$ vertices forms two concentric decagonal circles -- in the Coxeter plane $w$ acts as a rotation by $2\pi/5$ (as denoted by the two coloured vertices in a) with the Coxeter element taking one into the other), whilst in the plane $IB_C$ it acts as a rotation by $4\pi/5$, as shown in b). 
It is easy to check that $W$ can indeed be written as $W=\exp(\frac{\pi}{5}B_C)\exp(-\frac{2\pi}{5}IB_C)$. It is clear that taking the product of simple roots in the Coxeter element in a different order introduces overall minus signs as well as minus signs in the exponentials, so we will not worry about these from now on. 

$A_4$ is unusual in that the projection from 4D only yields two concentric circles in the Coxeter plane. In fact, it consists of two copies of $H_2$ (panel c)) with a relative factor of $\tau$. This is similar to the situation for $E_8$ and $H_4$, as explained in Fig. \ref{figE8}, since by removing four of the eight nodes one gets a diagram folding from $A_4$ to $H_2$. 
In \cite{DechantTwarockBoehm2011H3aff}, we were considering affine extensions of $H_2$ by adding a translation operator and taking the orbit under the compact group (panel d)). 
What is striking is that this $H_2^{aff}$ point set, i.e. an affine extension of the decagon after one unit translation, is very nearly the Coxeter projection of $D_6$, which is shown in panel e).

The situation for the other 4D groups is similar, as shown in Figure \ref{figCoxPlB4}, which shows for the groups $B_4$, $D_4$, $F_4$ and $H_4$ that the Coxeter element acts in the Coxeter plane $B_C$ as rotations by $\pm 2\pi/h$, and in the plane defined by $IB_C$ as $h$-fold rotations giving the other exponents algebraically. 
For $B_4$ one has exponents $\{1, 3, 5, 7\}$ as shown in panels a) and b) and the decomposition into eigenblades of the Coxeter element $W=\exp(\frac{\pi}{8}B_C)\exp(\frac{3\pi}{8}IB_C)$ reflects this.

\begin{figure}
	\begin{center}
	 		      \begin{tabular}{@{}c@{ }c@{ }}
						\begin{tikzpicture}
						\node (img) [inner sep=0pt,above right]
					{\includegraphics[width=5cm]{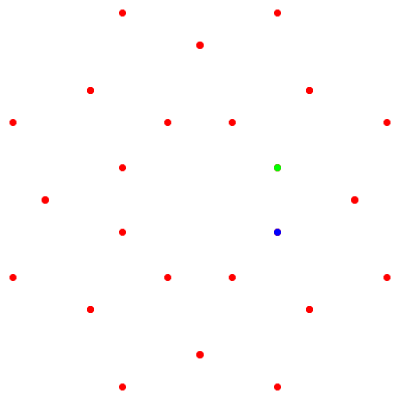}};
						\end{tikzpicture}&\hspace{1.5cm}
						\begin{tikzpicture}
						\node (img) [inner sep=0pt,above right]
					{\includegraphics[width=5cm]{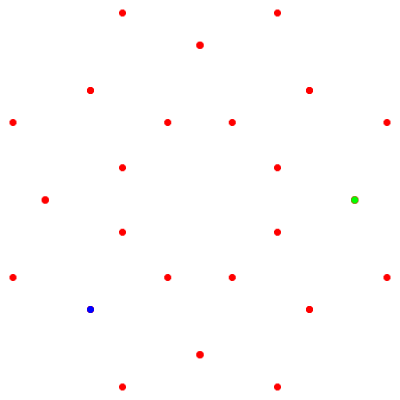}};
						\end{tikzpicture}
				\vspace{0.25cm}
			\\
					(a)&(b) \\
								\begin{tikzpicture}
								\node (img) [inner sep=0pt,above right]
							{\includegraphics[width=5cm]{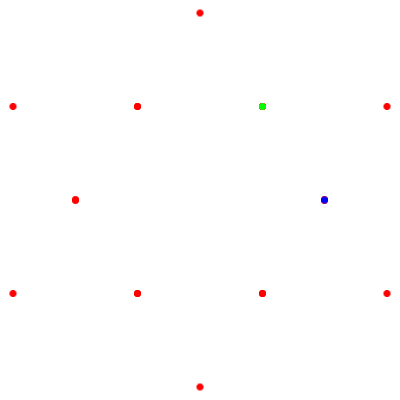}};
								\end{tikzpicture}&\hspace{1.5cm}
								\begin{tikzpicture}
								\node (img) [inner sep=0pt,above right]
							{\includegraphics[width=5cm]{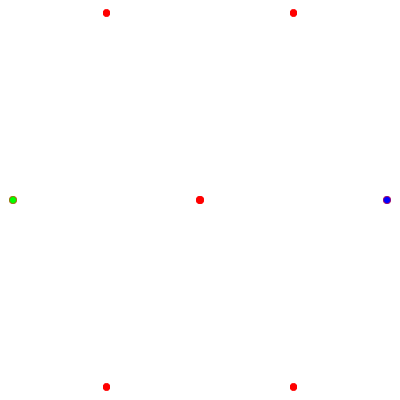}};
								\end{tikzpicture}
						\vspace{0.25cm}
					\\
							(c)&(d) \\
									\begin{tikzpicture}
									\node (img) [inner sep=0pt,above right]
								{\includegraphics[width=5cm]{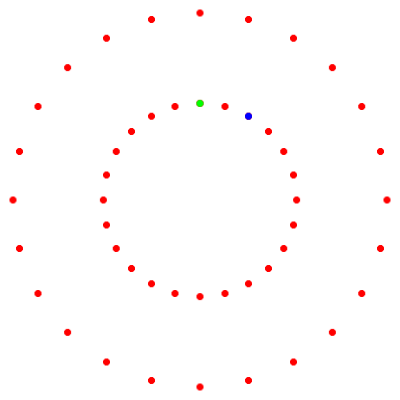}};
									\end{tikzpicture}&\hspace{1.5cm}
									\begin{tikzpicture}
									\node (img) [inner sep=0pt,above right]
								{\includegraphics[width=5cm]{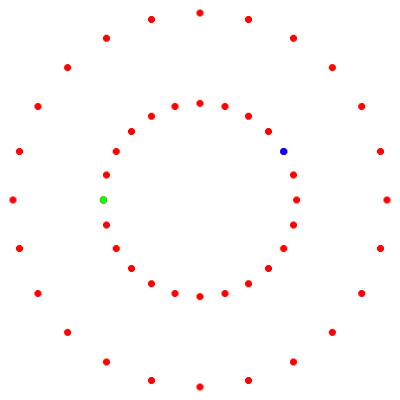}};
									\end{tikzpicture}
							\vspace{0.25cm}
						\\
								(e)&(f) \\
										\begin{tikzpicture}
										\node (img) [inner sep=0pt,above right]
									{\includegraphics[width=5cm]{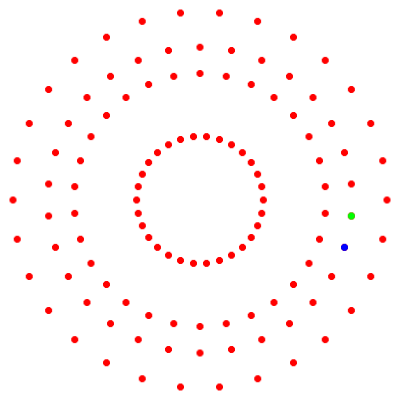}};
										\end{tikzpicture}&\hspace{1.5cm}
										\begin{tikzpicture}
										\node (img) [inner sep=0pt,above right]
									{\includegraphics[width=5cm]{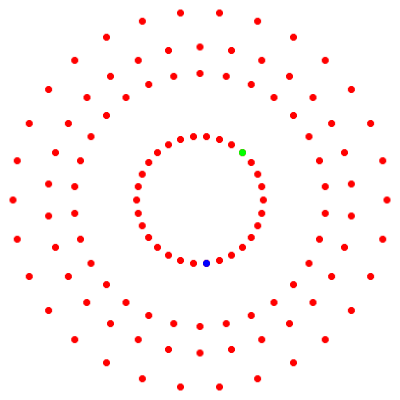}};
										\end{tikzpicture}
								\vspace{0.25cm}
							\\
									(g)&(h) \\
								
		\end{tabular}	
	
\end{center}
\caption[dummy1]{Coxeter projections of $B_4$, $D_4$, $F_4$ and $H_4$. In each row the plots are the action of the Coxeter element in the Coxeter plane given by $B_C$ and in the plane given by $IB_C$.}
\label{figCoxPlB4}
\end{figure}

{$D_4$} has {exponents $\{1, 3, 3, 5\}$} which is reflected in the fact that the Coxeter versor can be written as 
$W=\exp(\frac{-\pi}{6}B_C)\exp(\frac{\pi}{2}IB_C)=-\exp(\frac{-\pi}{6}B_C)IB_C$. The Coxeter projections of $D_4$ into the $IB_C$ plane demonstrate that the other factor in $W$ that does not come from $B_C$ is actually the product of two vectors ($e_1+e_2-2e_3$ and $e_1-e_2$ for simple roots $e_1$, $e_2$, $e_3$ and $\frac{1}{2}(e_4-e_1-e_2-e_3)$)
 rather than a bivector exponential (since the angle is $\pi/2$): on some vectors it acts as a rotation by $3=h/2$ in the plane, others it projects onto the origin (Fig. \ref{figCoxPlB4} panels c) and d)). 
{$F_4$} has {exponents $\{1, 5, 7, 11\}$} which again is evident from the Clifford factorisation $W=\exp(\frac{\pi}{12}B_C)\exp(\frac{5\pi}{12}IB_C)$ and the way it acts on the two planes (panels e) and f)). 
{$H_4$} has {exponents $\{1, 11, 19, 29\}$} and factorisation $W=\exp(\frac{\pi}{30}B_C)\exp(\frac{11\pi}{30}IB_C)$ (panels g) and h)). Table \ref{tab:1} summarises the factorisations of the 4D Coxeter versors. 

\begin{table}

%
%
\begin{tabular}{|c|c|c|c|}

	\hline
	rank 4 root system&$h$&exponents&W-factorisation
	\tabularnewline 	\hline 	\hline
	$A_4$&$5$&$1,2,3,4$&$W=\exp\left(\frac{\pi}{\textcolor{red}{5}}B_C\right)\exp\left(\frac{\textcolor{red}{2}\pi}{5}IB_C\right)$
	\tabularnewline	\hline
	$B_4$&$8$&$1,3,5,7$&$W=\exp\left(\frac{\pi}{\textcolor{red}{8}}B_C\right)\exp\left(\frac{\textcolor{red}{3}\pi}{8}IB_C\right)$
	\tabularnewline
		\hline
	$D_4$&$6$&$1,3,3,5$&$W=\exp\left(\frac{\pi}{\textcolor{red}{6}}B_C\right)\exp\left(\frac{\pi}{\textcolor{red}{2}}IB_C\right)$
		\tabularnewline
			\hline		$F_4$&$12$&$1,5,7,11$&$W=\exp\left(\frac{\pi}{\textcolor{red}{12}}B_C\right)\exp\left(\frac{\textcolor{red}{5}\pi}{12}IB_C\right)$
			\tabularnewline
				\hline
	$H_4$&$30$&$1,11,19,29$&$W=\exp\left(\frac{\pi}{\textcolor{red}{30}}B_C\right)\exp\left(\frac{\textcolor{red}{11}\pi}{30}IB_C\right)$
				\tabularnewline
	\hline
	\end{tabular}
	\caption{Summary of the factorisations of the Coxeter versors of the 4D root systems}
	\label{tab:1}       
\end{table}

The Coxeter versor of $D_6$ can be written as $W=\exp(\frac{\pi}{10}B_C)\exp(\frac{3\pi}{10}B_2)B_3$ for certain bivectors $B_2$ and $B_3$ from which it is evident that the 
exponents are indeed $\{1, 3, 5, 5, 7, 9\}$, including two reflections and otherwise 10-fold rotations in the Coxeter plane and another orthogonal plane (Figure Fig. \ref{figCoxPlD6} a) and c)). Since  as in Fig. \ref{figE8} by removing a pair of nodes  and as for $A_4\&H_2$ there is also a diagram folding from $D_6$ to $H_3$ ($H_3$ as we saw above has exponents $\{1, 5, 9\}$), the $D_6$ projection  again actually consists of two copies of that of $H_3$ with a relative factor of $\tau$ (panel b)) -- but the $H_3$ projection already has radii with a relative factor of $\tau$ such that two orbits of the $D_6$ projection  fall on top of each other.

\begin{figure}
	\begin{center}
	 		      \begin{tabular}{@{}c@{ }c@{ }c@{ }}
						\begin{tikzpicture}
						\node (img) [inner sep=0pt,above right]
					{\includegraphics[width=5cm]{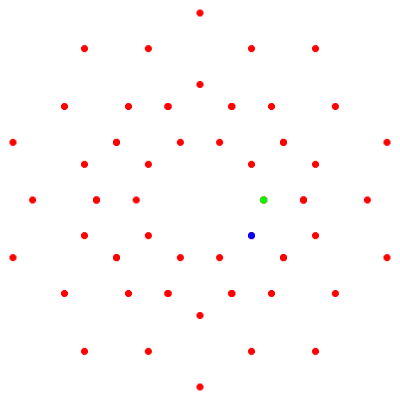}};
						\end{tikzpicture}&\hspace{0.5cm}
						\begin{tikzpicture}
						\node (img) [inner sep=0pt,above right]
							{\includegraphics[width=5cm]{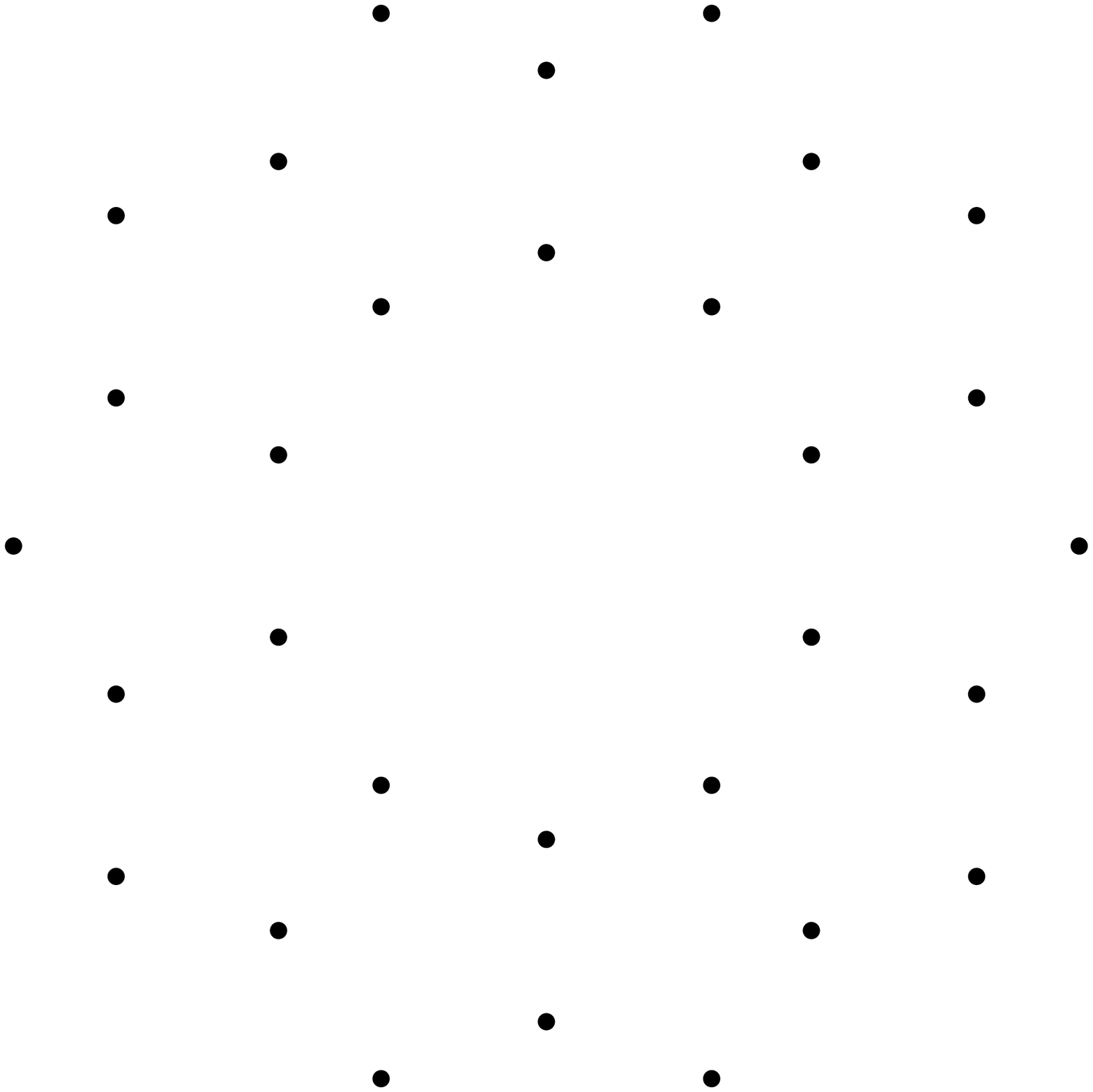}};
								\end{tikzpicture}&\hspace{0.5cm}
								\begin{tikzpicture}
								\node (img) [inner sep=0pt,above right]
					{\includegraphics[width=5cm]{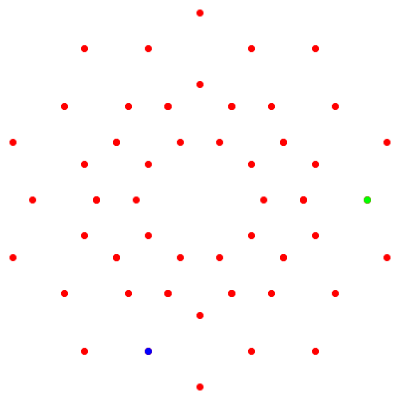}};
						\end{tikzpicture}
				\vspace{0.25cm}
			\\
					(a)&(b)&(c) \\

		\end{tabular}	
	
\end{center}
\caption[dummy1]{Coxeter projections of $D_6$ (panels a) and c)) and $H_3$ (panel b)). The two innermost radii of the $H_3$ projection are precisely in the ratio of $\tau$, such that since the $D_6$ projection consists of two copies of the $H_3$ with a relative radius of $\tau$, two orbits actually coincide. In the Coxeter plane a) and the other eigenplane c) the Coxeter element acts by $10$-fold rotation, as expected. }
\label{figCoxPlD6}
\end{figure}

\begin{figure}
	\begin{center}
	 		      \begin{tabular}{@{}c@{ }c@{ }c@{ }}
						\begin{tikzpicture}
						\node (img) [inner sep=0pt,above right]
					{\includegraphics[width=5cm]{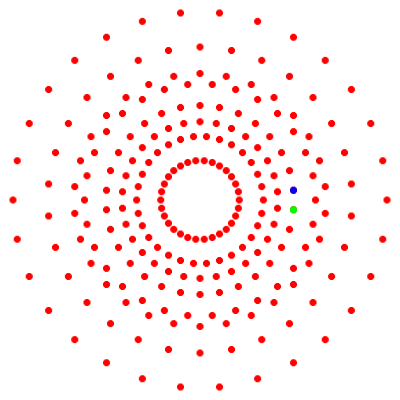}};
						\end{tikzpicture}&\hspace{0.5cm}
						\begin{tikzpicture}
						\node (img) [inner sep=0pt,above right]
						{\includegraphics[width=5cm]{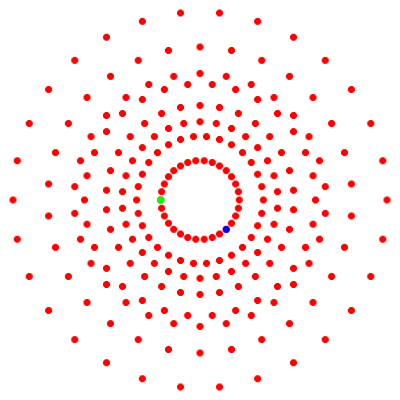}};
							\end{tikzpicture}&\hspace{0.5cm}
							\begin{tikzpicture}
							\node (img) [inner sep=0pt,above right]
					{\includegraphics[width=5cm]{H4CoxPl.png}};
						\end{tikzpicture}
				\vspace{0.25cm}
			\\
					(a)&(b)&(e)  \\
							\begin{tikzpicture}
							\node (img) [inner sep=0pt,above right]
						{\includegraphics[width=5cm]{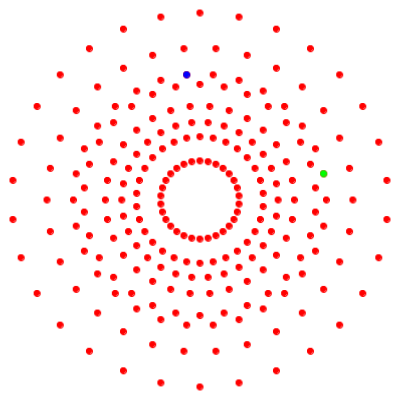}};
							\end{tikzpicture}&\hspace{0.5cm}
							\begin{tikzpicture}
							\node (img) [inner sep=0pt,above right]
							{\includegraphics[width=5cm]{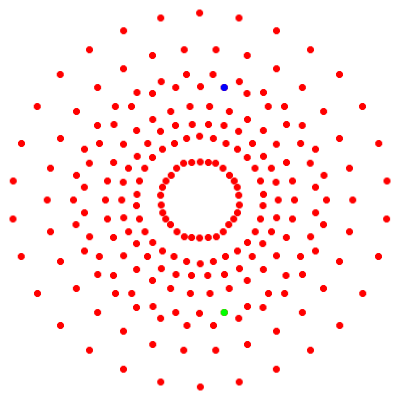}};
								\end{tikzpicture}&\hspace{0.5cm}		
					\vspace{0.25cm}
							\\
									(c)&(d)&  \\
		\end{tabular}	
	
\end{center}
\caption[dummy1]{Coxeter projections of $E_8$ and $H_4$ (c)). Again, $E_8$ consists of $H_4$ and $\tau H_4$ with the Coxeter element acting as a $30$-fold rotation in both, as usual in the Coxeter plane a) and according to the exponents $11$, $7$ and $13$ in b), c) and d), respectively. }
\label{figCoxPlE8}
\end{figure}

Not surprisingly, analogously the Coxeter versor for $E_8$ can be written as $$-W=\exp(\frac{\pi}{30}B_C)\exp(\frac{11\pi}{30}B_2)\exp(\frac{7\pi}{30}B_3)\exp(\frac{13\pi}{30}B_4).$$ This gives the well-known exponents $\{1, 7, 11, 13, 17, 19, 23, 29\}$ a more geometric meaning as $30$-fold rotations in four orthogonal eigenplanes (Fig. \ref{figCoxPlE8}). As we have alluded to in Section \ref{sec_spin}, the Coxeter projection of $E_8$ actually consists of two copies of that of $H_4$ in the bottom row of Figure \ref{figCoxPlB4} with a relative radius of $\tau$. We recall that $H_4$ also has exponents $\{1, 11, 19, 29\}$ and since it is a subgroup of $E_8$ they have of course the same Coxeter element and number,  and share two eigenplanes with exponents $1$ and $29$ as well as $11$ and $19$. 

For this $E_8$ example we explicitly give the details. 
We choose the simple roots of $E_8$ (in the Clifford algebra of 8D, rather than as earlier the 8D Clifford algebra of 3D) as
$$\alpha_1=\frac{1}{\sqrt{2}}(e_7-e_6), \alpha_2=\frac{1}{\sqrt{2}}(e_6-e_5), \alpha_3=\frac{1}{\sqrt{2}}(e_5-e_4), \alpha_4=\frac{1}{\sqrt{2}}(e_4-e_3),$$
$$\alpha_5=\frac{1}{\sqrt{2}}(e_3-e_2), \alpha_6=\frac{1}{\sqrt{2}}(e_2-e_1), \alpha_8=\frac{1}{\sqrt{2}}(e_2+e_1),$$ $$\alpha_7=\frac{1}{\sqrt{8}}(e_1-e_2-e_3-e_4-e_5-e_6-e_7+e_8).$$

Then the Coxeter versor $W=\alpha_2\alpha_4\alpha_6\alpha_8\alpha_3\alpha_5\alpha_1\alpha_7$ is given by $$W=-\frac{1}{16}+\frac{1}{16}\,e_{{2,3,4,5,6,7}}+\frac{1}{16}\,e_{{1,2,3,4,5,6,7,8}}-\frac{1}{16}\,e_{{1,4,5,6,7,8}}-\frac{1}{16}\,e_{{1,2,3,6,7,8}}-\frac{1}{16}\,e_{{4,5,6,7}}-\frac{1}{16}\,e_{{2,3,6,7}}$$
$$+\frac{1}{16}\,e_{{1,6,7,8}}+\frac{1}{16}\,e_{{1,2,5,6,7,8}}+\frac{1}{16}\,e_{{6,7}}
+\frac{1}{16}\,e_{{2,5,6,7}}+\frac{1}{16}\,e_{{1,3,5,6,7,8}}-\frac{1}{16}\,e_{{1,2,4,6,
7,8}}+\frac{1}{16}\,e_{{3,5,6,7}}-\frac{1}{16}\,e_{{2,4,6,7}}$$
$$-\frac{1}{16}\,e_{{1,3,4,6,7,8}}-
\frac{1}{16}\,e_{{3,4,6,7}}+\frac{1}{16}\,e_{{1,3,4,8}}+\frac{1}{16}\,e_{{3,4}}+\frac{1}{16}\,e_{{1,2,
3,4,7,8}}-\frac{1}{16}\,e_{{1,4,7,8}}+\frac{1}{16}\,e_{{2,3,4,7}}+\frac{1}{16}\,e_{{1,2,3,5,7,
8}}$$
$$-\frac{1}{16}\,e_{{4,7}}+\frac{1}{16}\,e_{{2,3,5,7}}-\frac{1}{16}\,e_{{1,5,7,8}}-\frac{1}{16}\,e_{{
5,7}}-\frac{1}{8}\,e_{{1,5,6,7}}+\frac{1}{16}\,e_{{1,5,6,8}}+\frac{1}{16}\,e_{{1,2,7,8}}+\frac{1}{16}
\,e_{{5,6}}\\+\frac{1}{16}\,e_{{1,3,7,8}}$$
$$+\frac{1}{16}\,e_{{3,7}}+\frac{1}{16}\,e_{{2,4,5,7}}+\frac{1}{16}\,e_{{1,3,4,5
,7,8}}+\frac{1}{8}\,e_{{1,3,4,7}}\\-\frac{1}{16}\,e_{{1,3,4,5,6,8}}+\frac{1}{16}\,e_{{3,4,5,7}}+
\frac{3}{8}\,e_{{1,3,4,5,6,7}}-\frac{1}{16}\,e_{{3,4,5,6}}$$
$$+\frac{1}{4}\,e_{{1,2,4,5,6,7}}-\frac{1}{16}
\,e_{{1,2,4,5,6,8}}-\frac{1}{16}\,e_{{1,2,6,8}}\\-\frac{1}{16}\,e_{{2,4,5,6}}+\frac{1}{8}\,e_{{1
,2,6,7}}-\frac{1}{16}\,e_{{1,2,3,5,6,8}}+\frac{1}{16}\,e_{{2,7}}+\frac{1}{8}\,e_{{1,7}}-\frac{1}{16}\,e_{{1,8}}$$
$$-\frac{1}{16}\,e_{{2,6}}+\frac{1}{4}\,e_{{1,2,3,5,6,7}
}-\frac{1}{16}\,e_{{2,3,5,6}}+\frac{1}{8}\,e_{{1,2,3,4,6,7}}-\frac{1}{16}\,e_{{1,2,3,4,6,8}}-\frac{1}{16}\,e_{{2,3,4,6}}+\frac{1}{16}\,e_{{1,2,4,8}}+\frac{1}{16}\,e_{{2,4}}$$
$$-\frac{1}{16}\,e_{{1,2,5
,8}}+\frac{1}{8}\,e_{{1,2,5,7}}-\frac{1}{16}\,e_{{2,5}}+\frac{1}{16}\,e_{{1,2,3,8}}+\frac{1}{16}\,e_{{
2,3}}-\frac{1}{16}\,e_{{1,2,3,4,5,8}}+\frac{1}{8}\,e_{{1,2,3,4,5,7}}+\frac{1}{16}\,e_{{1,4,6,8
}}-\frac{1}{16}\,e_{{2,3,4,5}}$$
$$+\frac{1}{16}\,e
_{{1,2,4,5,7,8}}-\frac{1}{4}\,e_{{1,4,6,7}}-\frac{1}{16}\,e_{{1,3,6,8}}+\frac{1}{16}\,e_
{{1,4,5,8}}+\frac{1}{4}\,e_{{1,3,5,7}}-\frac{1}{16}\,e_{{1,3,5,8}}-\frac{1}{8}\,e_{{1,4,5,6}}-
\frac{1}{8}\,e_{{1,3,5,6}}-\frac{1}{4}\,e_{{1,3,4,6}}$$
$$+\frac{1}{8}\,e_{{1,4}}+\frac{1}{8}\,e_{{1,3}}-\frac{1}{4}\,e_{{1,3,4,5}}-\frac{1}{8}\,e_{{1,2,4,6}}-\frac{1}{8}\,e_{{1,2,3,6}}-\frac{1}{8}\,e_{{1,2,4,
5}}-\frac{1}{8}\,e_{{1,2,3,5}}+\frac{1}{16}\,e_{{4,6}}+\frac{1}{16}\,e_{{4,5}}-\frac{1}{16}\,e_{{3,6}}
-\frac{1}{16}\,e_{{3,5}}$$
$$=-\exp(\frac{\pi}{30}B_C)\exp(\frac{11\pi}{30}B_2)\exp(\frac{7\pi}{30}B_3)\exp(\frac{13\pi}{30}B_4),$$
where the eigenplane bivectors $B_i$ are orthogonal to one another and are explicitly given by the outer products of the vector pairs (given numerically, for simplicity)

$$B_C=v_1\wedge w_1, \,\, v_1=2.8129\,e_{{3}}+ 2.8129\,e_{{4}}+ 4.5514\,e_{{5}}+
 4.5514\,e_{{6}}+ 5.1395\,e_{{7}}+ 21.771\,e_{{8}},$$
$$w_1=- 0.23919\,e_{{1}}+ 1.6534\,e_{{2}}+ 1.6534\,e_{{3}}+
 3.9417\,e_{{4}}+ 5.1113\,e_{{6}}+ 5.1113\,e_{{7}}+
 3.9417\,e_{{5}}+ 21.652\,e_{{8},}$$

$$B_2=v_2\wedge w_2, \,\, v_2=1.1504\,e_{{3}}+ 1.1504\,e_{{4}}+ 1.8614\,e_{{5}}+
 1.8614\,e_{{6}}- 0.2405\,e_{{7}}- 1.0188\,e_{{8}},$$
$$w_2=2.0904\,e_{{1}}- 0.6762\,e_{{2}}- 0.6762\,e_{{3}}+
 1.612\,e_{{4}}- 0.097822\,e_{{6}}- 0.097822\,e_{{7}}+
 1.612\,e_{{5}}- 0.4144\,e_{{8}},$$

$$B_3=v_3\wedge w_3, \,\, v_3=2.1019\,e_{{3}}+ 2.1019\,e_{{4}}- 1.2991\,e_{{5}}-
 1.2991\,e_{{6}}- 4.112\,e_{{7}}+ 0.97071\,e_{{8}},
$$
$$w_3=- 0.58484\,e_{{1}}+ 1.9991\,e_{{2}}+ 1.9991\,e_{{3}}+
 1.125\,e_{{4}}- 3.0558\,e_{{6}}- 3.0558\,e_{{7}}+
 1.125\,e_{{5}}+ 0.72138\,e_{{8}},
$$
and
$$B_4=v_4\wedge w_4, \,\, v_4=0.58806\,e_{{3}}+ 0.58806\,e_{{4}}- 0.36344\,e_{{5}}-
0.36344\,e_{{6}}+ 0.78698\,e_{{7}}- 0.18578\,e_{{8}},
$$
$$w_4=0.85493\,e_{{1}}+ 0.55928\,e_{{2}}+ 0.55928\,e_{{3}}-
0.31475\,e_{{4}}+ 0.16362\,e_{{6}}+ 0.16362\,e_{{7}}-
0.31475\,e_{{5}}- 0.038626\,e_{{8}}.
$$
In Clifford algebra, the Coxeter versor as the product of the simple roots therefore completely encodes the factorisation into orthogonal eigenplanes, as of course it must, as Geometric Algebra and the root system completely determine the geometry, without the need for artificial complexification.

\section{Conclusions}\label{concl}
We have shown that with the help of Clifford algebra \textbf{all} exceptional root systems can in fact be constructed from the 3D root systems alone. 
This offers a revolutionarily new way of viewing these phenomena in terms of spinorial geometry of 3D as intrinsically 3D phenomena, with huge potential implications for the many areas in which these symmetries appear.
Likewise, the geometry of the Coxeter plane is best viewed in a Clifford algebra framework, which provides geometric meaning and insight, for instance the complex eigenvalues in standard theory are just seen to be rotations in eigenplanes of the Coxeter element given by bivectors which act as unit imaginaries.
Since the definition of root systems stipulates a vector space with an inner product, the associated Clifford algebra of this space can easily be constructed without any loss of generality and is actually the most natural framework to use for such root systems and Coxeter groups, both since Clifford algebra affords a uniquely simple reflection formula and since combining the linear structure of the space with the inner product is better than using both separately. 

\section*{Acknowledgements}\label{ack}
I would like to thank David Hestenes and his wife Nancy for their support over the years and for inviting me to spend half a year with them at ASU; my PhD advisor Anthony Lasenby and his wife Joan Lasenby for all their support over more than a decade;  Eckhard Hitzer for his constant support and encouragement, and likewise Sebasti\`a Xamb\'o, in particular for organising a splendid conference and for suggesting (with David Hestenes) a revised and more detailed writeup for the proceedings.

\bibliographystyle{plain}

\end{document}